\documentclass[10pt,letterpaper]{article}

\usepackage{amsmath}		%used for \eqref, \langle, \rangle, environment \begin{equation*}, \DeclareMathOperator, \begin{align}, \begin{cases}, ...
\usepackage{amssymb}        %\mathbb
\usepackage{amsthm}		%\theoremstyle (for "remarks" environment), \qed, \begin{proof}, 

\usepackage{color,graphicx} 	%eps, xFig figures
	\graphicspath{{figuras/}}
\usepackage{enumerate} 		%used for \begin{enumerate}[i)]
\usepackage{hyperref}
\usepackage{multicol}	%for \begin{multicols}{2}
\newtheorem{theorem}{Theorem} 
\newtheorem{proposition}{Proposition}[section]

\newtheorem{definition}{Definition}
\newtheorem{lemma}[proposition]{Lemma}
{\theoremstyle{remark} 		%usa amsthm
	  \newtheorem*{remark*}{Remark} 
} 

\DeclareMathOperator{\cof}{cof}

\newcommand{\dd}{\, \mathrm{d}}
\DeclareMathOperator{\Det}{Det}

\DeclareMathOperator{\dist}{dist}       %amsmath
\renewcommand{\div}{\operatorname{div}}
\newcommand{\eps}{\varepsilon}

\DeclareMathOperator{\imT}{im_T}

\DeclareMathOperator{\loc}{loc}
\newcommand{\N}{\mathbb{N}}             %amssymb

\newcommand{\R}{\mathbb{R}}             %amssymb

\newcommand{\weakc}{\rightharpoonup}
\newcommand{\weakcs}{\overset{*}{\rightharpoonup}}

\begin{document}

\begin{center}
  \thispagestyle{empty}
  \vspace*{6em}
  {\Large
    {\LARGE Multiple independent cavitation in 2D neoHookean materials}
    \\[5em]
    
    by \\[1em]
    Victor Andr\'es Rodolfo Ca\~nulef Aguilar
    \\[5em]
    
    Thesis presented to the Faculty of Mathematics of the \\
    Pontificia Universidad Cat\'olica de Chile\\
    for the degree of\\
    Master in Mathematics
    \\[5em]
    
    15 May 2017\\
    Santiago, Chile}
 
\end{center}

\pagebreak

\section{Introduction and main result}

This work builds on preliminary (unpublished) results obtained jointly by Duvan Henao (Pontificia Universidad
Cat\'olica de Chile), who has 
guided this Masters project, and Sylvia Serfaty (New York University).

Cavitation in solid mechanics is the name given to the sudden formation and expansion 
of cavities in the interior of an elastic (or elasto-plastic) body subject to 
sufficiently large and multiaxial tension. 
The first experimental studies of cavitation in elastomers are due to 
Gent \& Lindley \cite{GeLi59}, who were also able to give
a theoretical prediction for the critical hydrostatic load at which 
the internal rupture occurs by solving the equilibrium equations for
an infinitely thick nonlinearly elastic shell under the assumption of radial symmetry.
The first analysis of the evolution of a cavity beyond its nucleation was due to Ball \cite{Ball82}, 
who showed that the one-parameter family of deformations
\begin{align} \label{eq:Ball-radial}
	u(x)=\sqrt[n]{|x|^n+L^n} \frac{x}{|x|},\quad L\geq 0,\quad n=2,3
\end{align}
provides a stable branch of weak solutions to the incompressible elasticity equations 
that bifurcates from the homogeneous deformation at the critical dead-load
predicted by Gent \& Lindley.
The assumption of radial symmetry, which persisted in this pioneering work,
was finally removed by M\"uller \& Spector \cite{MuSp95} and Sivaloganathan
\& Spector \cite{SiSp00} who proved the existence of minimizers of
the elastic energy allowing for all sorts of cavitation configurations.
Lopez-Pamies, Idiart \& Nakamura \cite{LoIdNa11} and 
Negr\'on-Marrero \& Sivaloganathan \cite{NeSi12} discussed the onset of cavitation
under non-symmetric loadings and Mora-Corral \cite{Mora14} studied the quasistatic evolution of cavitation. We refer to the Introduction in \cite{HeSe13} and the 
references therein for a more complete guide through the extensive literature on this fracture mechanism.

This thesis is concerned with the determination of the maximum load at which the cavities formed no longer grow independently, retaining their spherical symmetry, but are forced to interact with each other.
 The interaction between cavities, in Sobolev models for perforated domains, has been numerically studied in 
 \cite{XuHe11,LianLi11,LianLi11JCPAM,LianLi12,Lefevre15}.
Henao \& Mora-Corral \cite{HeMo10} proposed a free-discontinuity model allowing 
for fracture by void coalescence, which was further analyzed in \cite{HeMo11,HeMo12,HeMo15}. 
An Ambrosio-Tortorelli regularization of this model was presented in \cite{HeMoXu15} 
and implemented in \cite{HeMoXu16}, showing the transition from the independent growth of circular cavities to coalescence. However, the only existing quantitative analysis of the interaction between cavities is due to Henao \& Serfaty \cite{HeSe13}
who study the behaviour of
\begin{align}
  \nonumber
 \Upsilon(\eps, a_1, & a_2, v_1, v_2) := \min \Bigg\{  \int_{{\mathcal{B}} \setminus (B(a_1,\eps)\cup B(a_2,\eps))}
    \frac{|Du|^2}{2} dx\,:\ 
 \\ 
    & u\in H^1({\mathcal{B}} \setminus (B(a_1,\eps)\cup B(a_2,\eps)); \R^2),\ 
	u(x)=\lambda x\ \text{for }x\in\partial{\mathcal{B}},
	\\ \nonumber  & u \text{ is invertible (in a certain sense),\ } \det Du= 1\ \text{a.e.},
	\\ \nonumber & \text{and}\ |\imT(u, B(a_i,\eps))|=v_i +O(\eps^2),\ i=1,2.
	\Bigg \}
\end{align}
as the puncture scale $\eps$ and the distance $|a_1-a_2|$ tend to zero.
Here ${\mathcal{B}}_\eps:={\mathcal{B}}\setminus (B(a_1,\eps)\cup B(a_2,\eps))$ can be
interpreted as
the reference configuration of a two-dimensional elastic body
containing initial micro-cavities centred at $a_1$ and $a_2$;
the map $u:{\mathcal{B}}_\eps
\to \R^2$ represents the deformation of the body subject to
the pure displacement boundary condition $u(x)=\lambda x$ 
on $\partial {\mathcal{B}}$, with $\lambda>1$;
for $i=1,2$
the set $\imT(u,B(a_i,\eps))$ is what is known 
in the literature \cite{Sverak88,MuSp95} as
the topological image of the ball $B(a_i,\eps)$,
and corresponds (in this case) to the space occupied by  cavity $i$
after the deformation;
$|\imT(B(a_i,\eps))|$ is the area of that cavity;
the positive parameters $v_1$ and $v_2$ are fixed 
(independent of $\eps$); finally,
the last constraint expresses that the deformed cavities
are required to have areas that are closer and closer
to $v_1$ and $v_2$, respectively, as $\eps\to0$.

The incompressibility constraint $\det Du\equiv 1$
imposes a relation between the stretch factor $\lambda$
associated to the Dirichlet condition on the outer
boundary $\partial {\mathcal{B}}$ and $v_1$ and $v_2$.
Indeed, 
\begin{align}
  \lambda^2 |{\mathcal{B}}| - ( v_1+v_2)
  &= \lim_{\eps \to 0} \left |\lambda {\mathcal{B}} 
      \setminus \big ( \imT(u_\eps, B(a_1,\eps)) \cup \imT(u_\eps, B(a_2,\eps)) \big )
      \right |
      \\ & = \lim_{\eps\to 0} \left | u_\eps
      \Big ( {\mathcal{B}} \setminus (B(a_1,\eps)\cup B(a_2,\eps)) \Big )
      \right |
      \\ &= \lim_{\eps\to0} \int_{{\mathcal{B}} \setminus (B(a_1,\eps)\cup B(a_2,\eps))} \det Du_\eps(x) \dd x
      = |{\mathcal{B}}|,
\end{align}
provided that effectively $|\imT(u_\eps, B(a_i,\eps))|=v_i +O(\eps^2),\ i=1,2$. We write $\eps\to0$ and $u_\eps$ for simplicity of notation, though in reality we are considering a sequence $\eps_j\to 0$, 
a corresponding sequence of deformations $(u_{\eps_j})_{j\in\N}$, and the limit as $j\to 0$.
Also, note that 
the constraint $|\imT(u, B(a_i,\eps))|=v_i$, with $v_i$ independent of $\eps$, cannot be satisfied because
the areas of the micro-cavities $B(a_i,\eps)$ need to be taken into account, hence the need of the $O(\eps^2)$.

The first result in \cite{HeSe13} (see \cite[Thm.\ 1.5]{HeSe13}) is the lower bound
\begin{multline}
   \int_{{\mathcal{B}} \setminus (B(a_1,\eps)\cup B(a_2,\eps))}
    \frac{|Du|^2}{2} dx
    \\
    \geq  \left ( \sum_{i=1}^2 |\imT(u, B(a_i,\eps))|\right )\left (C+ \log \frac{\dist(\{a_1,a_2\},\partial {\mathcal{B}})}{\eps}\right ),
\end{multline}
which is satisfied for any $u$ in the admissible space and any cavitation points $a_1$, $a_2$
(assuming they are fixed with respect to $\eps$).
This estimate shows, in particular, that the Dirichlet energy of any sequence $(u_\eps)_\eps$ blows
up as $(v_1+v_2)|\log\eps|$ when $\eps\to 0$ (at least in the case when $a_1$ and $a_2$ remain far from $\partial {\mathcal{B}}$),
which, in a sense, is to be expected since the singularity in the gradient of a map creating a cavity from a single point $a\in {\mathcal{B}}$
is at least of the order of 
  \begin{align*}
    |Du(x)|\sim \frac{L}{r},
    \quad
    r=|x-a|
\end{align*}
where $L$ is such that $ \pi L^2$ equals  the area of the created cavity. 
Nevertheless, one can still ask under what conditions the energy of a sequence $(u_\eps)_\eps$ blows up
at no more than the stated rate of $(v_1+v_2)|\log \eps|$, i.e., under what conditions 
the \emph{renormalized energy}
\begin{align}
    \int_{{\mathcal{B}} \setminus (B(a_1,\eps)\cup B(a_2,\eps))}
    \frac{|Du_\eps|^2}{2} dx - (v_1+v_2)|\log \eps| 
\end{align}
is uniformly bounded with respect to $\eps$.
A more general situation was considered in \cite{HeSe13}, where $a_1=a_{1,\eps}$ and $a_2=a_{2,\eps}$, as well 
as the ratio $\frac{v_{1,\eps}}{v_{2,\eps}}$ (but not the total cavity area $v_{1,\eps}
+v_{2,\eps}$),
are allowed to change with $\eps$.
It was proved in \cite[Thm.\ 1.9]{HeSe13} that 
if the renormalized energy is bounded independently of $\eps$
and the centres $a_{1,\eps}$, $a_{2,\eps}$ are compactly contained in ${\mathcal{B}}$
then, passing to a subsequence,
one of the following holds:
  \begin{enumerate}[i)]
   \item the sequences $(v_{1,\eps})_\eps$ and $(v_{2,\eps})_\eps$
   converge to values $v_1$ and $v_2$ that are strictly positive; 
   the cavities $\imT(u_\eps, B(a_{1_\eps},\eps))$ and $\imT(u_\eps, B(a_{2,\eps},\eps))$
   converge to disks of areas $v_1$ and $v_2$ (in the metric given by $\dist(E_1,E_2)
   =|E_1\triangle E_2|$); and (under the additional assumption that the midpoints
   $\frac{a_{1,\eps}+a_{2,\eps}}{2}$ remain far from $\partial {\mathcal{B}}$),
   the distance $|a_{1,\eps}-a_{2,\eps}|$ does not vanish as $\eps\to0$.
   
   \item One of the sequences $(v_{i,\eps})_\eps$  (say  $(v_{2,\eps})_\eps$)  vanishes
   as $\eps\to 0$ and the cavities $\imT(u_\eps, B(a_{1,\eps},\eps))$ converge to
   a disk of area $v_1$.
   
   \item The distances $|a_{1,\eps}-a_{2,\eps}|$ scale like $O(\eps)$
   as $\eps\to 0$; the unions of the cavities $\imT(u_\eps, B(a_{1,\eps},\eps)) \cup \imT(u_\eps, B(a_{2,\eps},\eps))$
   converge to a disk of area $v_1+v_2$; and each of the cavities, independently,
   are necessarily distorted, in the sense that their distance to the set of all disks 
   is bounded away from zero.
  \end{enumerate}
  
  Here we consider a more restrictive setting where the centres $a_1$ and $a_2$ 
  in the reference configuration, as well as the target cavity areas $v_1>0$ and $v_2>0$, 
  are given (they are part of the data of the problem, together with ${\mathcal{B}}$). 
  In particular, for any sequence $(u_\eps)_\eps$ with bounded renormalized energy, 
  scenario ii) -where one of the cavities closes up in the limit- and scenario iii) -where the cavities are
  pushed together to form one equivalent round cavity- will not occur;
 we will only be left with the possibility that 
 the cavities $\imT(u_\eps, B(a_{1},\eps))$ and $\imT(u_\eps, B(a_{2},\eps))$
  must converge to disks of areas $v_1$ and $v_2$.
  We interpret this as saying that the second stage in the experimental observations of fracture inititation
  in elastomers, in which the cavities formed stop growing independently (retaining their spherical symmetry)
  and start deforming  together (to the point of eventually coalescing), corresponds to a 
  higher energy regime, it is not attainable with an energy of just $(v_1+v_2)|\log \eps|$.
  On the other hand, it is not always possible to produce circular ``independent'' cavities of any given areas $v_1$ and $v_2$
  coming from any fixed locations $a_1$ and $a_2$ in the reference configuration.
  Indeed, suppose that $\imT(u_\eps, B(a_{1},\eps))$ and $\imT(u_\eps, B(a_{2},\eps))$
  effectively converge to disks $E_1$ and $E_2$, of areas $v_1$ and $v_2$, which must be contained
  in $\lambda{\mathcal{B}}$. If, for example, ${\mathcal{B}}=B(0, R_0)$, we must have that the line segment joining
  the centres of $E_1$ and $E_2$ is shorter than (or equal to) the radius $\lambda R_0$,
  so necessarily
  $$ \sqrt{\frac{v_1}{\pi}} + \sqrt{\frac{v_2}{\pi}}\leq \lambda R_0.$$
  On the other hand, incompressibility yields
  $$v_1+v_2=(\lambda^2-1)\pi R_0^2.$$
  Hence, a necessary condition for the existence of a sequence $(u_\eps)_\eps$
  with bounded renormalized energy is that 
  $ (\lambda^2-1)R_0^2 + \frac{2\sqrt{v_1v_2}}{\pi} \leq \lambda^2R_0^2$, i.e.
  $$2\sqrt{v_1v_2}\leq \pi R_0^2.$$
  This shows that if the load is sufficiently large (if the requirement is imposed that cavities 
  must be opened of areas $v_1$ and $v_2$ with $2\sqrt{v_1v_2} > \pi R_0^2$)
  then the deformations must necessarily enter in the higher energy regime 
  where the energies blow up at a rate higher than $(v_1+v_2)|\log \eps|$.
  This gives rise to the question of for what values of $v_1$ and $v_2$ and what locations $a_1$ and $a_2$
  can the hypothesis of the existence of a sequence with bounded renormalized energy 
  actually be satisfied. This is the specific question we address in this article. 
  
  In fact,
  we consider a more general version of the above mentioned question where the material
  can open not only two but an arbitrarily large (albeit fixed) number of cavities.
  We consider the case of a circular domain ${\mathcal{B}}$ and of a displacement 
  condition of the form $u(x)=\lambda x$ for $x$ on the outer boundary $\partial {\mathcal{B}}$,
  though in reality more general domains and Dirichlet conditions could be treated
  with minor modifications from this work.
  We prove that a sufficient condition on $a_1$, $a_2$, \ldots, $a_n$
  and $v_1$, $v_2$, \ldots, $v_n$, for a given $n\in \N$, for the existence of a sequence of deformations
  $(u_\eps)_\eps$ with bounded renormalized energy is that 
  the following simple geometric property be satisfied.
  
\begin{definition}
    \label{de:attainable}
    Let $n\in \N$, $R_0>0$, and ${\mathcal{B}}:=B(0,R_0)\subset \R^2$. We say that $\Big ( (a_i)_{i=1}^n, (v_i)_{i=1}^n \Big )$ 
    is a configuration attainable 
    through an evolution of circular cavities (or, more briefly, an \emph{attainable configuration})
    if $a_i \in{\mathcal{B}}$ and $v_i>0$ for all $i\in \{1,\ldots, n\}$,
    and  there exist evolutions
    \begin{itemize}
     \item $z_i\in C^1([1,\lambda], \R^2)$ of the cavity centres, and
     \item $L_i:[1,\lambda]\to [0,\infty)$ of the cavity radii,
    \end{itemize}
    where $\lambda$ is given by $$\sum_{i=1}^n v_i = (\lambda^2-1)\pi R_0^2,$$
    such that 
    \begin{align} \label{eq:incLi}
     \sum_{i=1}^n \pi L_i^2(t) = (t^2-1)\pi R_0^2\qquad \forall\, t\in [1,\lambda]
    \end{align}
     and for each $i\in \{1,\ldots, n\}$
    \begin{enumerate}[i)]
    \item $L_i^2$ belongs to $C^1([1,\lambda],[0,\infty))$;
     \item $z_i(1)=a_i$ and $L_i(1)=0$;
     \item $\pi L_i^2(\lambda)=v_i$; and 
     \item for all $t\in [1,\lambda]$ the disks $\overline{B(z_i(t), L_i(t))}$ are disjoint 
     and contained in $B(0, tR_0)$.
    \end{enumerate}
\end{definition}

Although other time parametrizations are of course possible for the evolution of the centres
and the radii in the above definition, we have chosen the stretch factor at the outer boundary
$\partial {\mathcal{B}}$ as our parameter.

\begin{theorem} \label{th:main}
  Let $n\in \N$ and ${\mathcal{B}}=B(0,R_0)\subset \R^2$. Suppose that the configuration 
  $\Big ( (a_i)_{i=1}^n, (v_i)_{i=1}^n \Big )$ is attainable. 
  Let $\eps_j\to 0$ be a sequence that we will denote in what follows simply by $\eps$. 
  Set ${\mathcal{B}}_\eps:= {\mathcal{B}} \setminus \bigcup_{i=1}^n \overline{B}_\eps(a_i)$.
  Assume that for every $\eps$ the map $u_\eps$ minimizes $\int_{{\mathcal{B}}_\eps} |Du|^2\dd x$
  among all $u\in H^1({\mathcal{B}}_\eps;\R^2)$ satisfying
  \begin{itemize}
   \item the invertibility condition (INV) of Definition \ref{de:INV};
   \item $u(x)=\lambda x$ for $x\in \partial {\mathcal{B}}$;
   \item $\det Du(x)=1$ for a.e.\ $x\in{\mathcal{B}_\eps}$;
   \item and $|\imT(u, B_{\eps}(a_i))|=v_i +O(\eps^2)$ for all $i\in \{1,\ldots, n\}$.
  \end{itemize}
  Then there exists a constant $C=C\big (n, R_0, (a_i)_{i=1}^n, (v_i)_{i=1}^n\big )$
  independent of $\eps$ such that 
  $$ \int_{{\mathcal{B}}_\eps} \frac{|Du_\eps|^2}{2}\dd x \leq 
  C + \left ( \sum_{i=1}^n v_i \right ) |\log \eps|.$$
  Moreover, there exists a subsequence (not relabelled) and 
  $u\in \bigcap_{1\leq p < 2} W^{1,p}({\mathcal{B}},\R^2) \cap H^1_{\loc}({\mathcal{B}}\setminus \{a_1, \ldots, a_m\},\R^2)$
  such that 
  \begin{itemize}
   \item $u_\eps\weakc u$ in $H^1_{\loc} ({\mathcal{B}}\setminus \{a_1, \ldots, a_m\},\R^2)$;
   \item $\Det Du_\eps \weakcs \Det Du$ in ${\mathcal{B}}\setminus \{a_1, \ldots, a_m\}$;
   locally in the sense of measures (where $\Det Du$ is the distributional Jacobian of 
   Definition \ref{de:DetDu});
   \item $\Det Du = \sum_{i=1}^n v_i \delta_{a_i} + \mathcal L^2$ in ${\mathcal{B}}$ (where $\mathcal L^2$ is
   the Lebesgue measure);
   \item The cavities $\imT(u, a_i)$ (as defined in Definition \ref{de:imT}) are disks of area $v_i$, for all $i\in \{1,\ldots,n\}$;
    \item $|\imT(u_\eps, B_{\eps}(a_i)) \triangle \imT(u,a_i)|\to 0$ as $\eps\to 0$ for $i\in\{1,\ldots,n\}$.
  \end{itemize}

\end{theorem}

The following example gives a sense about which configurations $\Big ( (a_i)_{i=1}^n, (v_i)_{i=1}^n \Big )$
are attainable through an evolution of circular cavities.

\begin{proposition}
  Let $n\in \N$, $a_1,\ldots, a_n\in {\mathcal{B}}:=B(0,R_0)\subset \R^2$,  $v_1,\ldots, v_n>0$.
  Let $\lambda>1$ be such that $(\lambda^2-1)\pi R_0^2 = \sum v_i$. Set 
  \begin{align}
    \sigma = \min \left \{ \min_ i \frac{\left ( 1-\frac{|a_i|}{R_0}\right)^2}{\frac{v_i}{\sum v_k}},
    \min_{i\ne j} \frac{ |a_i-a_j|^2}{R_0^2 \left ( \sqrt{\frac{v_i}{\sum v_k}} + \sqrt{\frac{v_j}{\sum v_k}} \right )^2}
    \right \}.
  \end{align}
  Then both in the case $\sigma\geq 1$ and in the case $\sigma<1$ and $\lambda^2 < \frac{1}{1-\sigma}$
  the configuration $\Big ( (a_i)_{i=1}^n, (v_i)_{i=1}^n \Big )$
is attainable through an evolution of circular cavities.
\end{proposition}

\begin{proof}
  For every $t\in [1,\lambda]$ and every $i\in \{1,\ldots, n\}$ set 
  \begin{align}
    z_i(t):=ta_i, \quad L_i(t):=\sqrt{(t^2-1)\frac{v_i}{\sum v_k}}\cdot R_0.
  \end{align}
  We only need to check that the $\overline{B(z_i(t),L_i(t))}$ are disjoint and contained in $B(0,tR_0)$ for all $t$
  (the remaining conditions in Definition \ref{de:attainable} are immediately verified).
  Both in the case $\sigma\geq 1$ and in the case $\sigma<1$ and $\lambda^2 < \frac{1}{1-\sigma}$
  we have that 
  $$ 1-\lambda^{-2} < \sigma.$$
  As a consequence, we obtain that 
  $$1-t^{-2}<\sigma\quad \forall\,t\in[1,\lambda].$$
  Hence, $$1-t^{-2} < \frac{\left ( 1-\frac{|a_i|}{R_0}\right)^2}{\frac{v_i}{\sum v_k}}\quad \forall\,i $$
  and 
  $$1-t^{-2} < \frac{ |a_i-a_j|^2}{R_0^2 \left ( \sqrt{\frac{v_i}{\sum v_k}} + \sqrt{\frac{v_j}{\sum v_k}} \right )^2}
  \quad \forall\,i\ne j.$$
  It is easy to see that the first inequality is equivalent to $$ L_i(t)^2 < t^2(R_0-|a_i|)^2$$
  which in turn says that $L_i(t)+|z_i(t)| < tR_0$ (i.e., each $\overline{B(z_i(t), L_i(t))}\subset B(0,tR_0)$).
  Analogously, the second inequality is equivalent to $$(\sqrt{L_i(t)}+\sqrt{L_j(t)})^2< t^2|a_i-a_j|^2$$
  which in turn says that $L_i(t)+L_j(t) < |z_i(t)-z_j(t)|$ (i.e., the disks are disjoint). 
  This completes the proof.
\end{proof}

\begin{remark*}
  In the case when $v_1=v_2=\cdots=v_n$,
  \begin{align}
   \sigma &= \frac{\displaystyle n\pi \min\left \{ \min_i (R_0-|a_i|)^2, \min_{i\ne j} \left ( \frac{|a_i-a_j|}{2}\right )^2 \right \}}{\pi R_0^2}.
  \end{align}
  This is the packing density of the largest disjoint collection of the form $\{ B(a_i,\rho): i\in\{1,\ldots,n\}\}$
  contained in ${\mathcal{B}}$. There is an extensive literature on the famous circle packing problem; 
  for example, it is known \cite{Melissen94} that when $n=11$ the maximum packing density is  
  $$ \frac{11}{\left (1+\frac{1}{\sin \frac{\pi}{9}}\right )^2} \approx 0.7145,$$
  which yields the upper bound 
  $$\lambda < \sqrt{\frac{(1+\sin \frac{\pi}{9})^2}{1+2\sin \frac{\pi}{9}- 10\sin^2\frac{\pi}{9}}}\approx 1.8714
  $$
  for which our above construction is able to produce attainable configurations with 11 cavities of equal size.
\end{remark*}

In Section \ref{se:notation} we introduce the notation used in the rest of this thesis and state some preliminary results.
In Section \ref{se:movingdomain} we investigate how does the regularity of the solution to a transport problem
depends on the geometry of the domain, with a view towards constructing an evolution of incompressible maps
in domains with circular holes that grow as the displacement boundary condition increases.
In Section \ref{se:proof} we put together the different arguments and prove Theorem \ref{th:main}.

\section{Notation and preliminaries}
\label{se:notation}

\subsection*{Green's function and function spaces}

\noindent $\Phi(x) :=\frac{-1}{2\pi}\log(|x|).$\\
$\Omega=\{x\in \mathbb{R}^2: R<|x|<R+d\}.$\\
$\Omega'=\{x\in \mathbb{R}^2: R+\frac{1}{3}d<|x|<R+\frac{2}{3}d\}.$\\
$C_{per}^{0,\alpha}=\{ g\in C_{loc}^{0,\alpha}(\mathbb{R}): g \text{ is $2\pi$-periodic} \}.$\\
$\phi^x(y)=\frac{1}{2\pi}ln(|y-x^{*}|)-\frac{|y|^2}{4\pi R^2}.$\\
$G_{N}(x,y)=\Phi (x)-\phi^x(y).$\\
$x^{*}=\frac{R^2}{|x|^2}x.$\\
$\left\|f\right\|_{\infty}=\sup|f(x)|.$\\
$ [f]_{0,\alpha}=\sup_{x\neq y}\frac{|f(x)-f(y)|}{|x-y|^{\alpha}}.$\\
$ [f]_{1,\alpha}=\sup_{x\neq y}\frac{|Df(x)-Df(y)|}{|x-y|^{\alpha}}.$\\
$ \left\|f\right\|_{0,\alpha}=\left\|f\right\|_{\infty}+[f]_{0,\alpha}.$\\
$ \left\|f\right\|_{1,\alpha}=\left\|f\right\|_{\infty}+\left\|Df\right\|_{\infty}+[f]_{1,\alpha}.$\\
$u_{,\beta}=\partial_{\beta}u .$\\

\subsection*{Assumptions on the geometry of the domain}
Throughout Section \ref{se:movingdomain} we will work in a generic domain with circular holes
\begin{align}\label{eq:genericE}
 E=B(z_0,r_0)\setminus\bigcup_{k=1}^{n}\overline{B(z_k,r_k)}.
\end{align}
The notation $d$ will be reserved for a generic length that controls (from below) the distance
between holes, their radii and the distance from them to the exterior boundary $\partial \mathcal B$,
i.e., $E$ is assumed to be such that 
\begin{align} \label{eq:d}
  \begin{gathered}\dist(\partial B(z_j,r_j),\partial B(z_k,r_k))\geq 2d\quad \forall\,j\ne k\in \{0,1,\ldots, n\},\\ 
r_i\geq Cd  \text{ for each } i\in \{1,\ldots, n\}\text{, and } r_0\geq C_0d 
\text{ for some } C_0>1.
\end{gathered}
\end{align}

\subsection*{Poincar\'e constant}

The Poincar\'e constant (for the Neumann problem) shall be denoted by $C_P$:
$$C_P(E):= \inf \left \{ \|\phi\|_{L^2(E)}: \phi \in H^1(E) \text{ such that } \|D\phi\|_{L^2(E)}=1
\ \text{and}\ \int_E \phi =0\right \}.$$

Given $\delta >0$  we denote by $\mathcal{F}_{\delta}$ 
the class of all domains of the form $E=B_0\setminus \bigcup_{i=1}^{n}B(z_i,r_i)$,
for some $n\in \mathbb{N}$, $r_0,r_1, \cdots r_n>0$, and $z_0,\cdots, z_n\in \mathbb{R}^{2}$, 
such that $\forall\thinspace i\geq 1 
\thinspace B(z_i,r_i)\subset B(z_0,r_0)$, 
$\forall \thinspace i\neq j \thinspace \overline{B(z_i,r_i)}$ and $\overline{B(z_j,r_j)}$ are disjoint, 
and $F(E)\geq \delta$, where
$$F(E):=\frac{1}{2r_0}\min\{\underset{i\neq j }\min \medspace dist(\partial B_i,\partial B_j), 
\medspace \underset{i}\min r_i \}.$$

\subsection*{Topological image and condition INV}

We give a succint definition of the topological image (see \cite{HeSe13} for more details).
\begin{definition}
 \label{de:imT}
 Let $u\in W^{1,p}(\partial B(x,r), \R^2)$ for some $x\in \R^2$, $r>0$, and $p>1$. Then
 $$\imT(u, B(x,r)):=\{y\in \R^2: \deg (u, \partial B(x,r), y)\ne 0\}.$$
\end{definition}

Given $u\in W^{1,p}(E, \R^2)$ and $x\in E$, there is a set $R_x\subset (0,\infty)$,
 which coincides a.e.\ with $\{r>0: B(x,r)\subset E\}$, such that $u|_{\partial B(x,r)}\in W^{1,p}$
 and both $\deg (u, \partial B(x,r), \cdot )$ and $\imT(u, B(x,r))$ are well defined for all $r\in R_x$.

 \begin{definition}
 \label{de:INV}
 We say that $u$ satisfies condition INV if for every $x\in E$ and every $r\in R_x$
 \begin{enumerate}[(i)]
  \item $u(z)\in \imT(u, B(x,r))$ for a.e.\ $z\in B(x,r)\cap E$ and
  \item $u(z)\in \R^2 \setminus \imT(u,B(x,r))$ for a.e.\ $z\in E\setminus B(x,r)$.
 \end{enumerate}
\end{definition}
 
 If $u$ satisfies condition INV then $\{\imT(u,B(x,r)): r\in R_x\}$ is increasing in $r$ for every $x$.
 
 \begin{definition}
 Given $a\in E$ we define $$\imT(u,a):= \bigcap_{r\in R_a} \imT(u, B(a,r)).$$
 Analogously, if $u\in W^{i,p}$ is defined and satisfies condition INV in a domain of the form $E=\mathcal B
 \setminus \bigcup_1^n B(z_i, r_i)$, then 
 we define 
 $$\imT(u, B(z_i, r_i))= \bigcap_{\substack{r\in R_{z_i}\\ r>r_i}} \imT(u, B(z,r)).$$
 \end{definition}
 
\subsection*{Distributional Jacobian}
\begin{definition}
 \label{de:DetDu}
 Given $u\in W^{1,2}(E,\R^2)\cap L^{\infty}_{\text{loc}}(E,\R^2)$ 
 its distributional Jacobian is defined as the distribution 
 $$ \langle \Det Du, \phi\rangle:= -\frac{1}{2}\int_E u(x)\cdot (\cof Du(x))D\phi(x)\dd x,\quad \phi\in C_c^{\infty}(E).$$
\end{definition}

\section{H\"older regularity for a transport problem in a moving domain}
\label{se:movingdomain}

\begin{proposition} \label{prop1}
Let $v$ be harmonic in $\Omega$ and $\zeta$ be a cut-off function with support within $|x|<R+\frac{2}{3}d$ and equal to $1$ for $|x|\leq R+\frac{1}{3}d$, then, if $u=\zeta v$:

$$u(x)=C-\int_{\partial B_R}\frac{\partial u}{\partial \nu}\left( \Phi(y-x)-\phi^{x}(y) \right)dS(y)-\int_{\Omega}\Delta u \left( \Phi(y-x)-\phi^{x}(y)  \right)dy.$$ 

\end{proposition}

\textbf{Proof}: Let us proceed as in \cite{Evans10}:

$$ \int_{\Omega\setminus B_{\varepsilon}(x)}\Delta u(y)\Phi(y-x)-u(y)\Delta_{y}\Phi (y-x)  dy=\int_{\partial \Omega}\frac{\partial u}{\partial \nu}\Phi(y-x)-\frac{\partial \Phi}{\partial \nu}(y-x)u(y) dS(y)$$
$$+\int_{\partial B_{\varepsilon}(x)}\frac{\partial \Phi}{\partial \nu}(y-x)u(y)-\frac{\partial u}{\partial \nu}\Phi(y-x)dS(y),$$
letting $\varepsilon \rightarrow 0$ (and using the fact that $u$ vanishes outside $B_{R+\frac{2}{3}d}$), we get:
$$ \int_{\Omega}\Delta u(y)\Phi(y-x) dy= \int_{\partial B_R}\frac{\partial \Phi}{\partial \nu}(y-x)u(y)-\frac{\partial u}{\partial \nu}\Phi(y-x) dS(y)-u(x).$$
Hence:$$ u(x)=\int_{\partial B_R}\frac{\partial \Phi}{\partial \nu}(y-x)u(y)-\frac{\partial u}{\partial \nu}\Phi(y-x) dS(y)-\int_{\Omega}\Delta u(y)\Phi(y-x) dy,$$
with the normal pointing outside $B_R$. Now (as can be seen in \cite{DiBenedetto09}), note that if a function $\phi^{x}(y)$ satisfies:

\begin{align} \label{benedetto}
\left\{ \begin{aligned}
-\Delta_{y}\phi^{x}(y)&=k& &\text{if $y\in \Omega$,} \\
\frac{\partial \phi^{x}}{\partial \nu} &=\frac{\partial \Phi}{\partial \nu}(y-x) & &\text{if $y\in \partial B_R$ ,}
\end{aligned} \right. \end {align}

\noindent being $k$ a constant, then:
$$ \int_{\Omega}\Delta_{y}\phi^{x}(y)u(y)-\Delta u \phi^{x}(y)dy=\int_{\partial \Omega}u(y)\frac{\partial}{\partial \nu}\phi^{x}(y)-\phi^{x}(y)\frac{\partial u}{\partial \nu} dS(y)$$
$$=\int_{\partial B_R}\phi^{x}(y)\frac{\partial u}{\partial \nu}-u(y)\frac{\partial}{\partial \nu}\Phi(y-x) dS(y)=k\int_{\Omega}u dy -\int_{\Omega}\Delta u \phi^{x}(y)dy,$$
\noindent where we have used \eqref{benedetto}. Finally, replacing in the expression for $u(x)$, we obtain:
$$ u(x)=C-\int_{\partial B_R}\frac{\partial u}{\partial \nu}\left( \Phi(y-x)-\phi^{x}(y) \right)dS(y)-\int_{\Omega}\Delta u \left( \Phi(y-x)-\phi^{x}(y)  \right)dy  .$$
It is easy to see that $\phi^x(y)=\frac{1}{2\pi}\log(|y-x^{*}|)-\frac{|y|^2}{4\pi R^2}$ satisfies \eqref{benedetto} using the identity $|x_1||x_2-x_1^{*}|=|x_2||x_1-x_2^{*}|$.\\

\noindent The following regularity estimates for harmonic functions can be found in \cite[Thm.\ 2.2.7]{Evans10}

\begin{lemma}   \label{harm reg}
Let $v$ be weakly harmonic in $B(x,d)$, then:\\
$ \left\|v\right\|_{L^{\infty}(B(x,\frac{d}{2}))}\leq C d^{-2}\left\|v\right\|_{L^1(B(x,d))} .$\\
$ \left\|D^{\beta}v\right\|_{L^{\infty}(B(x,\frac{d}{2}))}\leq C d^{-2-|\beta|}\left\|v\right\|_{L^1(B(x,d))} .$

\end{lemma}

\begin{proposition} \label{prop2}
: Let $v$ be harmonic in the distributional sense in $\Omega$ and $R\geq Cd$, then we have the folllowing estimates :\\
 $ \left\|v\right\|_{L^{\infty}(\Omega')}\leq C d^{-2}\left\|v\right\|_{L^1(\Omega)} .$\\
$ [v  ]_{0,\alpha(\Omega')}\leq Cd^{-3}R^{1-\alpha}\left\|v\right\|_{L^1(\Omega)}.$\\
$ \left\|D^{\beta}v\right\|_{L^{\infty}(\Omega')}\leq C d^{-2-|\beta|}\left\|v\right\|_{L^1(\Omega)} .$\\
$ [ v ]_{1,\alpha(\Omega')}\leq Cd^{-4}R^{1-\alpha}\left\|v\right\|_{L^1(\Omega)}.$

\end {proposition}

\textbf{Proof}: The first and third estimates follow from the previous Lemma.
To prove the second estimate note that using polar coordinates we get (for $r\in (R+\frac{1}{3}d,R+\frac{2}{3}d)$ and $\theta_1,\theta_2\in [-\pi,\pi]$, such that $|\theta_1-\theta_2|\leq \pi$):
$$ |v(re^{i\theta_1})-v(re^{i\theta_2})|\leq \int_{\theta_1}^{\theta_2}\left|\frac{d}{d\theta}\left(v(re^{i\theta})\right)\right|d\theta \leq \int_{\theta_1}^{\theta_2}\left|\frac{\partial v}{\partial x_1}\right|r|\sin(\theta)|+\left|\frac{\partial v}{\partial x_2}\right|r|\cos(\theta)|d\theta$$
$$\leq Cd^{-3}\left\|v\right\|_{L^{1}(\Omega)}r|\theta_1-\theta_2|\leq Cd^{-3}\left\|v\right\|_{L^{1}(\Omega)}|re^{i\theta_1}-re^{i\theta_2}|^{\alpha}R^{1-\alpha},$$
\noindent since $r|\theta_1-\theta_2|\leq \frac{\pi}{2}|re^{i\theta_1}-re^{i\theta_2}|$ (recall that $\frac{2}{\pi^2}\leq\frac{1-\cos(\theta)}{\theta^2} \leq\frac{1}{2}$, for $\theta\in [-\pi,\pi] $) and $|re^{i\theta_1}-re^{i\theta_2}|\leq 2r\leq CR$ .\\
Moreover, for $\theta\in [-\pi,\pi]$ and $r_1,r_2\in[R+\frac{1}{3}d,R+\frac{2}{3}d]$, we have:
$$|v(r_1e^{i\theta})-v(r_2e^{i\theta})|\leq \int_{r_1}^{r_2}\left|\frac{d}{dr}\left(v(re^{i\theta})\right)\right|dr\leq \int_{r_1}^{r_2}\left|\frac{\partial v}{\partial x_1}\right||\cos(\theta)|+\left|\frac{\partial v}{\partial x_2}\right||\sin(\theta)|dr$$
$$\leq Cd^{-3}\left\|v\right\|_{L^{1}(\Omega)}|r_1-r_2|\leq Cd^{-3}\left\|v\right\|_{L^{1}(\Omega)}|r_1e^{i\theta}-r_2e^{i\theta}|^{\alpha}R^{1-\alpha}$$

Now, for $r_1,r_2\in[R+\frac{1}{3}d,R+\frac{2}{3}d]$, $r_1\leq r_2$ and $\theta_1,\theta_2\in [-\pi,\pi]$, such that $|\theta_1-\theta_2|\leq \pi$, we have:
$$ |v(r_1e^{i\theta_1})-v(r_2e^{i\theta_2})|\leq |v(r_1e^{i\theta_1})-v(r_1e^{i\theta_2})|+|v(r_1e^{i\theta_2})-v(r_2e^{i\theta_2})|$$
$$\leq Cd^{-3}R^{1-\alpha}\left\|v\right\|_{L^{1}(\Omega)}(|r_1e^{i\theta_1}-r_1e^{i\theta_2}|^{\alpha}+|r_1e^{i\theta_2}-r_2e^{i\theta_2}|^{\alpha})$$
$$\leq Cd^{-3}R^{1-\alpha}\left\|v\right\|_{L^{1}(\Omega)}(|r_1e^{i\theta_1}-r_2e^{i\theta_2}|^{\alpha}+|r_1e^{i\theta_1}-r_2e^{i\theta_2}|^{\alpha}),$$
\noindent since $|r_1e^{i\theta_1}-r_2e^{i\theta_2}|^{2}=(r_1-r_2)^2+2r_1r_2(1-\cos(\theta_1-\theta_2))\geq 2r_1^{2}(1-\cos(\theta_1-\theta_2))=|r_1e^{i\theta_1}-r_1e^{i\theta_2}|^2$ and $|r_1e^{i\theta_1}-r_2e^{i\theta_2}|\geq|r_1-r_2|$.
The proof of the fourth estimate is analogous.

\begin{lemma} \label{lemma1}

Let $R\geq Cd$, $v$ be harmonic in $\Omega$ and $\zeta$ a cut-off function with support within $|x|<R+\frac{2}{3}d$ and equal to $1$ for $|x|\leq R+\frac{1}{3}d$, then:\\
$ [\Delta(v\zeta)]_{0,\alpha(\mathbb{R}^2)}\leq CR^{1-\alpha}d^{-5}\left\|v\right\|_{L^1(\Omega)}. $\\
$ \left\|\Delta(v\zeta)\right\|_{\infty (\mathbb{R}^2)}\leq Cd^{-4}  \left\|v\right\|_{L^1(\Omega)} .$

\end {lemma}

\textbf{Proof}: It is clear that we can choose $\zeta$ to be such that: $|D^{k}\zeta|\leq C_kd^{-k}$ (and then $[\zeta]_{k,\alpha(\Omega')}\leq C_{k+1}d^{-k-1}R^{1-\alpha} $ since $\zeta\in C_{c}^{\infty}(B(0,R+d))$). Then, using Proposition \ref{prop1} and the estimates for $\zeta$ we get:

$$|\Delta(v\zeta)|\leq 2|\nabla v \cdot \nabla \zeta| +|v\Delta \zeta|\leq C d^{-4}\left\|v\right\|_{L^1(\Omega)}.$$
On the other hand:
$$ [\Delta(v\zeta)]_{0,\alpha(\Omega')}\leq 2[\nabla v \cdot \nabla \zeta]_{0,\alpha(\Omega')} +[v\Delta \zeta]_{0,\alpha(\Omega')}.$$
Now note that:
$$ [v_{,\beta} \cdot \zeta_{,\beta}]_{0,\alpha(\Omega')}\leq  [v_{,\beta}]_{0,\alpha(\Omega')}\left\|\zeta_{,\beta}\right\|_{\infty(\Omega')}+[\zeta_{,\beta}]_{0,\alpha(\Omega')}\left\|v_{,\beta}\right\|_{\infty(\Omega')}$$
$$\leq Cd^{-4}R^{1-\alpha}\left\|v\right\|_{L^1(\Omega)}\cdot d^{-1}+ Cd^{-2}R^{1-\alpha}\cdot d^{-3}\left\|v\right\|_{L^1(\Omega)},$$

\noindent furthermore:

$$[v\Delta \zeta]_{0,\alpha(\Omega')}\leq [v]_{0,\alpha(\Omega')}\left\|\Delta\zeta\right\|_{\infty(\Omega')}+[\Delta\zeta]_{0,\alpha(\Omega')}\left\|v\right\|_{\infty(\Omega')}$$
$$ \leq Cd^{-3}R^{1-\alpha}\left\|v\right\|_{L^1(\Omega)}\cdot d^{-2}+Cd^{-3}R^{1-\alpha}\cdot d^{-2}\left\|v\right\|_{L^1(\Omega)}.$$

\noindent Hence:
$$ [\Delta(v\zeta)]_{0,\alpha(\Omega')}\leq Cd^{-5}R^{1-\alpha}\left\|v\right\|_{L^1(\Omega)}.$$

\noindent Now if $x\in \Omega'$ and $y\in \mathbb{R}^2\setminus \overline{\Omega'}$, there exists $t\in (0,1)$ such that $z=tx+(1-t)y\in \partial\Omega'$, then we have
$$ |\Delta(v\zeta)(x)-\Delta(v\zeta)(y)|\leq |\Delta(v\zeta)(x)-\Delta(v\zeta)(z)|+|\Delta(v\zeta)(z)-\Delta(v\zeta)(y)|$$
$$=|\Delta(v\zeta)(x)-\Delta(v\zeta)(z)|\leq CR^{1-\alpha}d^{-5}\left\|v\right\|_{L^1(\Omega)}|x-z|^{\alpha}$$
$$=CR^{1-\alpha}d^{-5}\left\|v\right\|_{L^1(\Omega)}(1-t)^{\alpha}|x-y|^{\alpha}\leq CR^{1-\alpha}d^{-5}\left\|v\right\|_{L^1(\Omega)}|x-y|^{\alpha}$$
(Clearly if $x,y\in \mathbb{R}^2\setminus \overline{\Omega'}$, $|\Delta(v(x)\zeta(x))-\Delta(v(y)\zeta(y))|=0$).\\
Finally, we get:
$$[\Delta(\zeta v)]_{0,\alpha(\mathbb{R}^2)}\leq CR^{1-\alpha}d^{-5}\left\|v\right\|_{L^1(\Omega)}. $$

\begin{proposition} \label{prop3}

Let $f\in C_c^{0,\alpha}(\Omega')$, $R\geq Cd$ and $u=\int_{\mathbb{R}^2}f(y)\Phi(x-y)dy$, then:\\
$ \left\|D u\right\|_{\infty(\mathbb{R}^2)}\leq CR \left\|f\right\|_{\infty}  .$\\
$[Du]_{0,\alpha (B(0,R+d)\setminus\overline{B(0,R)})}\leq CR^{1-\alpha}\left\|f\right\|_{\infty}$\\
$  \left\|\partial_{\beta\gamma}^{2} u\right\|_{\infty(B(0,R+d)\setminus\overline{B(0,R)})}\leq  CR^{\alpha}[f]_{0,\alpha(\mathbb{R}^2)}+\frac{\delta_{\beta\gamma}}{2}\left\|f\right\|_{\infty} .$\\
$  [D^2 u]_{0,\alpha(B(0,R+d)\setminus\overline{B(0,R)})}\leq C[f]_{0,\alpha(\mathbb{R}^2)}. $

\end {proposition}

\textbf{Proof}: Let us estimate the first derivative:

$$|u_{,\beta}|\leq \left\|f\right\|_{\infty}\int_{\Omega'}\frac{dy}{|x-y|}\leq C\left\|f\right\|_{\infty}\int_{0}^{2R+\frac{5}{3}d}dr\leq CR\left\|f\right\|_{\infty}   ,$$

\noindent Now, let us estimate the Hold\"er seminorm of the derivatives: Let
$$v_{\rho}=\int_{\mathbb{R}^2\setminus B(x,\rho)}f(y)\Phi_{,\beta}(x-y)dy,$$
 with $\rho\in (0,2(R+d))$, then: $$|u_{,\beta}-v_{\rho}|\leq C\left\|f\right\|_{\infty}\int_{ B(x,\rho)}|x-y|^{-1}dy\leq C\left\|f\right\|_{\infty}\int_{ B(x,\rho)}|x-y|^{-1}dy$$
$$ \leq C\left\|f\right\|_{\infty}\rho\leq C\left\|f\right\|_{\infty}\rho^{\alpha}R^{1-\alpha}.$$
On the other hand:

$$\frac{\partial v_{\rho}}{\partial \gamma}=\int_{\mathbb{R}^2\setminus B(x,\rho)}f(y)\Phi_{,\beta\gamma}(x-y)dy-\int_{\partial B(x,\rho)}f(y)\Phi_{,\beta}(x-y)\nu_{\gamma}dS(y) ,$$
therefore: 
$$ \left|\frac{\partial v_{\rho}}{\partial \gamma}\right|\leq C\left\|f\right\|_{\infty}\left(  \int_{\mathbb{R}^2\setminus B(x,\rho)}|x-y|^{-2}dy+\int_{\partial B(x,\rho)}|x-y|^{-1}dS(y)\right)$$
$$\leq  C\left\|f\right\|_{\infty}\left(1+  \int_{B(x,2(R+d))\setminus B(x,\rho)}|x-y|^{-2}dy \right) $$
$$\leq C\left\|f\right\|_{\infty}\left(1+  \left|\log\left(\frac{R}{\rho}\right)\right| \right)\leq  C\left\|f\right\|_{\infty}\left(1+ \left(\frac{R}{\rho}\right)^{1-\alpha} \right).$$
(Note that $\frac{R}{\rho}\in(\frac{1}{2},\infty)$). Finally, if $|x-y|=\rho$:
$$|u_{,\beta}(x)-u_{,\beta}(y)|\leq |u_{,\beta}(x)-v_{\rho}(x)|+|v_{\rho}(x)-v_{\rho}(y)|+|v_{\rho}(y)-u_{,\beta}(y)|$$
$$\leq C\left\|f\right\|_{\infty}\rho^{\alpha}R^{1-\alpha}+C|x-y|\left\|f\right\|_{\infty}\left(1+ \left(\frac{R}{\rho}\right)^{1-\alpha}\right)$$
$$\leq C\left\|f\right\|_{\infty}\rho^{\alpha}R^{1-\alpha},$$
where we have used that $\rho\leq CR$.\\

\noindent To prove the third estimate, first note that the second derivatives of $u$ are given by:
$$u_{,\beta\gamma}=\lim_{\rho\rightarrow 0^+}\int_{\mathbb{R}^2\setminus B(x,\rho)}\Phi_{,\beta\gamma}(x-y)f(y)dy-\frac{\delta_{\beta\gamma}}{2}f.$$
\noindent Since $f\in C_{c}^{0,\alpha}$ (and using the fact that $\int_{\partial B(0,1)}\Phi_{,\beta\gamma}(z)dS(z)=0$, and $\int_{A}\Phi_{,\beta\gamma}(z)dz= 0$ if $A$ is any annulus centered at the origin ), the absolute value of the singular integral is bounded by:
$$  \left| \lim_{\rho\rightarrow 0^+}\int_{B(x,2R+\frac{5}{3}d)\setminus B(x,\rho)}(f(y)-f(x))\Phi_{,\beta\gamma}(x-y)dy \right|$$
$$\leq \lim_{\rho\rightarrow 0^+}\int_{\partial B(0,1)}|\Phi_{,\beta\gamma}(\omega)|dS(\omega)\int_{\rho}^{2R+\frac{5}{3}d}r^{\alpha-1}dr[f]_{0,\alpha}\leq CR^{\alpha}[f]_{0,\alpha},$$
that proves the second result (obviously we have $\left\|\frac{\delta_{ij}}{2}f\right\|_{\infty}\leq \frac{\delta_{ij}}{2}\left\|f\right\|_{\infty}$). To prove the last estimate, we proceed as in \cite[Thm.\ 2.6.4]{Morrey66}: first note that if $\Phi_{,ij}(x)=\Delta(x)$, $\omega(x)=u_{,ij}(x)+\frac{\delta_{ij}}{n}f(x)$, $n=2$, and 
$$\omega_{\rho}(x)=\int_{\mathbb{R}^n\setminus B(x,\rho)}\Delta(x-\xi)f(\xi)d\xi,$$
then:
$$|\omega_{\sigma}(x)-\omega_{\rho}(x)|\leq \int_{B(x,\rho)\setminus B(x,\sigma)}|\Delta(x-\xi)|[f]_{0,\alpha}|x-\xi|^{\alpha}d\xi\leq CM_0[f]_{0,\alpha}\rho^{\alpha},$$
being $M_0=\sup_{|x|=1}|\Delta(x)|$. If we let $\sigma \rightarrow 0$, we obtain:
$$|\omega(x)-\omega_{\rho}(x)|\leq CM_0[f]_{0,\alpha}\rho^{\alpha}.$$
Let $M=3R+3d$ and  $M_1=\sup_{|x|=1}|\nabla\Delta(x)|$. The derivatives of $\omega_{\rho}$ are given by:
$$\omega_{\rho,\beta}(x)=\int_{\mathbb{R}^n\setminus B(x,\rho)}\Delta_{,\beta}(x-\xi)f(\xi)d\xi-\int_{\partial B(x,\rho)}\Delta(x-\xi)f(\xi)d\xi_{\beta}^{'}$$
$$= \int_{B(x,M)\setminus B(x,\rho)}\Delta_{,\beta}(x-\xi)(f(\xi)-f(x))d\xi+\int_{\partial B(x,M)}\Delta(x-\xi)(f(\xi)-f(x))d\xi_{\beta}^{'}$$
$$+\int_{\partial B(x,\rho)}\Delta(x-\xi)(f(x)-f(\xi))d\xi_{\beta}^{'}$$
\noindent Note that: $$\int_{\partial B(x,M)}\Delta(x-\xi)f(\xi)d\xi_{\beta}^{'}=0.$$
\noindent Let $x,z\in B(0,R+d)$ and $\rho=|x-z|$,then:
$$|\nabla \omega_{\rho}|\leq C(M_0+M_1)[f]_{0,\alpha}(\rho^{\alpha-1}+M^{\alpha-1})\leq C(M_0+M_1)[f]_{0,\alpha}\rho^{\alpha-1}.$$
\noindent Thus (applying the mean value theorem):
$$|\omega(x)-\omega(z)|\leq |\omega(x)-\omega_{\rho}(x)|+|\omega_{\rho}(x)-\omega_{\rho}(z)|+|\omega_{\rho}(z)-\omega(z)|\leq C(M_0+M_1)[f]_{0,\alpha}\rho^{\alpha},$$
\noindent that yields: $[\omega]_{0,\alpha}\leq C(M_0+M_1)[f]_{0,\alpha}$.

\begin{lemma} \label{lemma2}

Let $R\geq Cd$ and $f\in C_c^{0,\alpha}(B_{R+\frac{2}{3}d}\setminus\overline{B_{R+\frac{d}{3}}})$, if $u=\int_{\mathbb{R}^2}f(y)\log|x^{*}-y|dy$, then:\\
$\left\| Du \right\|_{L^{\infty}(B_{R+d}\setminus \overline{B_R})}\leq C R\left\|f\right\|_{\infty}$.\\
$[ D u]_{0,\alpha(B_{R+d}\setminus \overline{B_R})} \leq CR^{2-\alpha}d^{-1}\left\|f\right\|_{\infty}$.\\
$\left\| D^2 u \right\|_{L^{\infty}(B_{R+d}\setminus \overline{B_R})} \leq CRd^{-1}\left\|f\right\|_{\infty}$.\\
$[ D^2 u]_{0,\alpha(B_{R+d}\setminus \overline{B_R})} \leq CR^{2-\alpha}d^{-2}\left\|f\right\|_{\infty}$.\\

\end {lemma}

\textbf{Proof}: Using the identity $|x_1||x_1^{*}-x_2|=|x_2||x_1-x_2^{*}|$, let us first note that:
\begin{align} \label{log-reflection}
\log|y-x^{*}|=\log|y^{*}-x|+\log|y|-\log|x|,  \end{align}
\noindent this implies that:
$$  u=C+\int_{\mathbb{R}^2}\log|x-y^{*}|f(y)dy-\log|x|\int_{\mathbb{R}^2}f(y)dy ,$$
\noindent then:
$$  |u_{,\beta}|\leq C\int_{\Omega'}\frac{|f(y)|dy}{|x-y^{*}|} + \frac{C}{|x|}\left\|f\right\|_{\infty}Rd\leq C\int_{\Omega'}\frac{|f(y)|dy}{|x|-|y^{*}|} + \frac{C}{|x|}\left\|f\right\|_{\infty}Rd$$
$$\leq CRd\frac{\left\|f\right\|_{\infty}}{R-\frac{R^2}{R+\frac{d}{3}}}+Cd\left\|f\right\|_{\infty}\leq CR\left\|f\right\|_{\infty}.$$
\noindent The other estimates are proved analogously (for the H\"older continuity we can use the same argument as in Proposition \ref{prop2}).

\begin{proposition} \label{prop4}

Let $f\in C_c^{0,\alpha}(B_{R+\frac{2}{3}d}\setminus\overline{B_{R+\frac{d}{3}}})$, $R\geq Cd$ and $u=\int_{\mathbb{R}^2}f(y)G_N(x,y)dy$, then (in $B_{R+d}\setminus \overline{B_R}$) :\\
$ \left\|Du\right\|_{\infty}\leq CR\left\|f\right\|_{\infty} .$\\
$ [D u]_{0,\alpha} \leq CR^{2-\alpha}d^{-1}\left\|f\right\|_{\infty}.$\\
$ \left\|D^2u\right\|_{\infty} \leq C(Rd^{-1}\left\|f\right\|_{\infty}+R^{\alpha}[f]_{0,\alpha}).$\\
$ [D^2 u]_{0,\alpha} \leq C(R^{2-\alpha}d^{-2}\left\|f\right\|_{\infty}+[f]_{0,\alpha}).$\\

\end {proposition}

\textbf{Proof}: It follows from Proposition \ref{prop3} and  Lemma \ref{lemma2}.

\begin{lemma} \label{lemma3}

Let $g\in C_{per}^{0,\alpha}$, $\phi\in[0,2\pi]$, $1<r_2<r_1$. 
Then: $$ |\omega(r_1e^{i\phi})-\omega(r_2e^{i\phi})|\leq Cr_1[g]_{0,\alpha}|r_1-r_2|^{\alpha},$$
\noindent where 

\begin{equation} \label{kernel1}
\omega:=\int_{-\pi}^{\pi}g(\tau+\phi)\frac{r\sin(\tau)d\tau}{r^2+1-2r\cos(\tau)}
\end{equation}

\end {lemma}

\textbf{Proof}: Note that:
$$|\omega(r_1e^{i\phi})-\omega(r_2e^{i\phi})|=\left|\int_{r_2}^{r_1}\frac{\partial \omega}{\partial r}dr\right|\leq \int_{r_2}^{r_1}\left|\frac{\partial \omega}{\partial r}\right|dr.$$
On the other hand:$$ \frac{\partial \omega}{\partial r}(re^{i\phi})=\int_{-\pi}^{\pi}g(\tau+\phi)\frac{(1-r^2)\sin(\tau)d\tau}{((1-r)^2+2r(1-\cos(\tau)))^2}$$
$$=\int_{-\pi}^{\pi}(g(\tau+\phi)-g(\phi))\frac{(1-r^2)\sin(\tau)d\tau}{((1-r)^2+2r(1-\cos(\tau)))^2},$$
where we have used that $\sin(\tau)$ is odd. Moreover:

$$\left|\int_{|\tau|\leq r-1}(g(\tau+\phi)-g(\phi))\frac{(1-r^2)\sin(\tau)d\tau}{((r-1)^2+2r(1-\cos(\tau)))^2}\right|$$
$$\leq\int_{|\tau|\leq r-1}\frac{2r_1(r-1)[g]_{0,\alpha}|\tau|^{1+\alpha}}{((r-1)^2+2r(1-\cos(\tau)))^2}\leq \int_{|\tau|\leq r-1}\frac{Cr_1[g]_{0,\alpha}(r-1)^{2+\alpha}}{(r-1)^4}d\tau$$
$$=Cr_1[g]_{0,\alpha}(r-1)^{\alpha-1}.$$

Recall that $\frac{2}{\pi^2}|\tau|^2\leq 1-\cos(\tau) \leq \frac{1}{2}|\tau|^2$ for $\tau\in (-\pi,\pi)$. To estimate the rest of the integral, it suffices to note that:

$$\left|\int_{ r-1 \leq|\tau|\leq \pi}(g(\tau+\phi)-g(\phi))\frac{(1-r^2)\sin(\tau)d\tau}{((r-1)^2+2r(1-\cos(\tau)))^2}\right|$$
$$\leq\int_{r-1\leq |\tau|\leq \pi}2r_1(r-1)[g]_{0,\alpha}\frac{|\tau|^{1+\alpha}}{((r-1)^2+2r(1-\cos(\tau)))^2}d\tau$$
$$\leq \int_{r-1\leq|\tau|\leq \pi}Cr_1(r-1)[g]_{0,\alpha}\frac{|\tau|^{1+\alpha}}{4|\tau|^4}d\tau \leq (r-1)Cr_1[g]_{0,\alpha}\int_{r-1\leq|\tau|\leq \pi}|\tau|^{\alpha-3}d\tau $$
$$\leq Cr_1(r-1)(r-1)^{\alpha-2}= Cr_1[g]_{0,\alpha}(r-1)^{\alpha-1}.$$

Finally:
$$|\omega(r_1e^{i\phi})-\omega(r_2e^{i\phi})|\leq \int_{r_2}^{r_1}\left|\frac{\partial \omega}{\partial r}\right|dr\leq Cr_1[g]_{0,\alpha}\int_{r_2}^{r_1}(r-1)^{\alpha-1}dr\leq Cr_1[g]_{0,\alpha}|r_1-r_2|^{\alpha}.$$
(Recall that $|x|^{\alpha}$ is locally H\"older continuous in $[0,\infty)$)

\begin{lemma} \label{lemma4}

Let $g\in C_{per}^{0,\alpha}$, $r>1$, $\omega$ as in \eqref{kernel1}, and $x_1,x_2\in\mathbb{R}^2$ such that $|x_1|=|x_2|=r$. Then:
$$|\omega(x_1)-\omega(x_2)|\leq Cr^2[g]_{0,\alpha}(r-1)^{\alpha-1}|x_1-x_2|.$$

\end {lemma}

\textbf{Proof}: Let $1<r\leq 2$ and $|\phi_1-\phi_2|\leq\pi$, if we define $K_r(\tau)=\frac{\sin(\tau)}{1+r^2-2r\cos(\tau)}$ then: 
$$\omega(re^{i\phi})=r\int_{-\pi}^{\pi}g(\tau+\phi)K_r(\tau)d\tau=-r\int_{-\pi}^{\pi}g(\tau)K_r(\phi-\tau)d\tau.$$ 

\noindent The derivative of $K_r$ is given by:
$$ \frac{\cos(\tau)(1+r^2)-2r}{(1+r^2-2r\cos(\tau))^2}=\left(1-\frac{(1+r)^2(1-\cos(\tau))}{(r-1)^2+2r(1-\cos(\tau))}\right)(1+r^2-2r\cos(\tau))^{-1}.$$

Since:

$$ \left|\frac{\cos(\tau)(1+r^2)-2r}{(r-1)^2+2r(1-\cos(\tau))}\right|\leq 1+\frac{(1+r)^2(1-\cos(\tau))}{2r(1-\cos(\tau))}\leq Cr,$$

\noindent we have:

$$\left|\frac{\partial K_r}{\partial \tau}(\tau)\right|\leq \frac{Cr}{(r-1)^2+2r(1-\cos(\tau))}\leq C'r|\tau|^{-2}, \text{if }|\tau|\leq \pi.$$

Let $\rho=|\phi_1-\phi_2|\leq \pi$, then:

$$\left|\frac{\partial \omega}{\partial \phi}\right|\leq r\left|\int_{-\pi}^{\pi}(g(\tau)-g(\phi))K_r'(\phi-\tau)d\tau\right|$$
$$ \leq    Cr^2[g]_{0,\alpha}\int_{|\tau-\phi|\leq r-1}\frac{|\tau-\phi|^{\alpha}}{(r-1)^2}d\tau+ Cr^2[g]_{0,\alpha}\int_{r-1\leq|\tau-\phi|\leq \pi}|\phi-\tau|^{\alpha-2}d\tau$$
$$\leq Cr^2(r-1)^{\alpha-1}[g]_{0,\alpha}. $$
Now using the fundamental theorem of calculus:
$$|\omega(re^{i\phi_1})-\omega(re^{i\phi_2})|\leq \int_{\phi_1}^{\phi_2}Cr^2(r-1)^{\alpha-1}[g]_{0,\alpha}d\phi$$
$$=Cr^2(r-1)^{\alpha-1}[g]_{0,\alpha}|\phi_1-\phi_2|\leq Cr^2(r-1)^{\alpha-1}[g]_{0,\alpha}|re^{i\phi_1}-re^{i\phi_2}|.$$

\begin{proposition} \label{prop5}

Let $g\in C_{per}^{0,\alpha}$, $\omega$ as in \eqref{kernel1}, and $x_1,x_2\in\mathbb{R}^2$ such that $ 1<|x_2|\leq|x_1|\leq 2$. Then:
$$|\omega(x_1)-\omega(x_2)|\leq C[g]_{0,\alpha}|x_1-x_2|^{\alpha}.$$
(i.e. $[\omega]_{0,\alpha}\leq C[g]_{0,\alpha}$).\\

\end {proposition}

\textbf{Proof}: Set $x_1=r_1e^{i\phi_1}$, $x_2=r_2e^{i\phi_2}$, $|\phi_1-\phi_2|\leq \pi$, $\rho:=|x_1-x_2|$\\
1. Case $r_1-1\geq \rho$: by lemmas \ref{lemma3} and \ref{lemma4} :
$$ |\omega(x_1)-\omega(x_2)|\leq |\omega(r_1e^{i\phi_1})-\omega(r_1e^{i\phi_2})|+|\omega(r_1e^{i\phi_2})-\omega(r_2e^{i\phi_2})|$$
$$\leq Cr_1[g]_{0,\alpha}(r_1-1)^{\alpha-1}|r_1e^{i\phi_1}-r_1e^{i\phi_2}|+Cr_1[g]_{0,\alpha}||x_1|-|x_2||^{\alpha}$$
$$  \leq  2C[g]_{0,\alpha}\rho^{\alpha-1}(|r_1e^{i\phi_1}-r_2e^{i\phi_2}|+|r_2e^{i\phi_2}-r_1e^{i\phi_2}|)+2C[g]_{0,\alpha}|x_1-x_2|^{\alpha} $$
$$  \leq   C[g]_{0,\alpha}(\rho^{\alpha-1}(\rho+\rho)+\rho^{\alpha})$$

2. Case $r_1-1< \rho$: Set $r:=1+\rho$. Note that since $r_2<r_1<2$, then $r=1+|x_1-x_2|<1+r_1+r_2\leq 5$ 
$$ |\omega(x_1)-\omega(x_2)|\leq |\omega(r_1e^{i\phi_1})-\omega(re^{i\phi_1})|+|\omega(re^{i\phi_1})-\omega(re^{i\phi_2})|+|\omega(re^{i\phi_2})-\omega(r_2e^{i\phi_2})|  $$
$$ \leq  2\cdot 5C[g]_{0,\alpha}|r-r_1|^{\alpha}+5C[g]_{0,\alpha}(r-1)^{\alpha-1}|re^{i\phi_1}-re^{i\phi_2}|, $$
since $r_2>1$, then $r-r_2=\rho-(r_2-1)<\rho$. On the other hand: $|re^{i\phi_1}-re^{i\phi_2}|\leq |r-r_1|+|x_1-x_2|+|r_2-r|<3\rho$
 and $(r-1)^{\alpha-1}=\rho^{\alpha-1}$ by definition of $r$. This completes the proof.

\begin{proposition} \label{prop6}

Let $g\in C_{per}^{0,\alpha}$, $\omega$ as in \eqref{kernel1}, and $x_1,x_2\in\mathbb{R}^2$ such that $ 1<|x_2|\leq|x_1|\leq 2$. Then:
$$ \left\|\omega\right\|_{\infty}\leq C[g]_{0,\alpha} .$$

\end {proposition}

\textbf{Proof}: It is easy to see that:

$$ |\omega|\leq C[g]_{0,\alpha}\int_{-\pi}^{\pi}\frac{|\tau|^{1+\alpha}}{|\tau|^2}d\tau\leq C[g]_{0,\alpha}.$$

\begin{lemma}  \label{lemma5}

Let $x=re^{i\phi}$ and $y=e^{i\tau}$. Let $u$ be given by:
 \begin{equation}\label{PKer}
u(re^{i\phi})=\frac{1-r^2}{2\pi}\int_{-\pi}^{\pi}\frac{g(\tau)d\tau}{|x-y|^2},
\end{equation}
then: $\left\|u\right\|_{\infty}\leq C\left\|g\right\|_{\infty}$\\

\end {lemma}

\textbf{Proof}: This is immediate from the well-known formula (see \cite{Gamelin01}):
 \begin{equation}\label{Poisson}
\frac{r^2-1}{2\pi}\int_{-\pi}^{\pi}\frac{d\tau}{1+r^2-2r\cos(\tau)}=sgn(r-1).
\end{equation}

\begin{lemma}   \label{lemma6}

Let $g\in C_{per}^{0,\alpha}$, $r>1$, $|\phi_1-\phi_2|\leq \pi$ and $u$ as in \eqref{PKer}. Then:
$$|u(re^{i\phi_1})-u(re^{i\phi_2})|\leq C[g]_{0,\alpha}|re^{i\phi_1}-re^{i\phi_2}|. $$

\end {lemma}

\textbf{Proof}: First note that (thanks to \eqref{Poisson}):

$$u(re^{i\phi})=\frac{1-r^2}{2\pi }\int_{-\pi}^{\pi}g(\tau)\frac{d\tau}{|x-y|^2}=\frac{1-r^2}{2\pi }\int_{-\pi}^{\pi}\frac{g(\tau+\phi)-g(\phi) }{1+r^2-2r\cos(\tau)}d\tau-g(\phi),$$
\noindent then:
$$|u(re^{i\phi_1})-u(re^{i\phi_2})|\leq [g]_{0,\alpha}|\phi_1-\phi_2|^{\alpha}+\frac{r^2-1}{2\pi }\int_{-\pi}^{\pi}\frac{|g(\tau+\phi_1)-g(\tau+\phi_2)|}{1+r^2-2r\cos(\tau)}d\tau$$
$$\leq [g]_{0,\alpha}|\phi_1-\phi_2|^{\alpha}+[g]_{0,\alpha}|\phi_1-\phi_2|^{\alpha}\frac{r^2-1}{2\pi}\frac{2\pi}{r^2-1}\leq C'[g]_{0,\alpha}|re^{i\phi_1}-re^{i\phi_2}|^{\alpha}.$$

\begin{lemma}    \label{lemma7}

Let $g\in C_{per}^{0,\alpha}$, $u$ as in \eqref{PKer}, $1<r_2<r_1\leq 2$. Then:
$$ |u(r_1e^{i\phi})-u(r_2e^{i\phi})|\leq C[g]_{0,\alpha}|r_1-r_2|^{\alpha}.$$

\end{lemma}

\textbf{Proof}: Note that:
$$ \frac{d}{dr}\left( \frac{1-r}{1+r^2-2r\cos(\tau)} \right)=\frac{(r-1)^2-2(1-\cos(\tau))}{((r-1)^2+2r(1-\cos(\tau)))^2},$$
also:
$$\frac{d }{dr}\left(\frac{(1+r)(1-r)}{(1-r)^2+2r(1-\cos(\tau))}\right)=(1+r)\frac{d}{dr}\left( \frac{1-r}{1+r^2-2r\cos(\tau)} \right)$$
$$+ \frac{1-r}{1+r^2-2r\cos(\tau)}. $$
We want to prove $\left|\frac{\partial u}{\partial r}\right|\leq C(r-1)^{\alpha-1}$, for $r\in (1,2)$. For that, it suffices to estimate the following integrals:
$$ \left|(r-1)\int_{-\pi}^{\pi}(g(\tau+\phi)-g(\phi))\frac{d\tau}{(r-1)^2+2r(1-\cos(\tau))}\right|\leq C\pi^{\alpha}[g]_{0,\alpha}(r-1)\frac{2\pi}{r^2-1}$$
$$\leq C[g]_{0,\alpha} \leq C[g]_{0,\alpha}(r-1)^{\alpha-1}.$$
Now let us estimate the second integral for $|\tau|\leq r-1$:
$$ 2\left|\int_{|\tau|\leq r-1}(g(\tau+\phi)-g(\phi))\frac{1-\cos(\tau)}{((r-1)^2+2r(1-\cos(\tau)))^2}d\tau\right|$$
$$\leq C[g]_{0,\alpha}\int_{|\tau|\leq r-1}\frac{|\tau|^{\alpha+2}}{((r-1)^2+2r(1-\cos(\tau)))^2}d\tau$$
$$\leq C[g]_{0,\alpha}\int_{|\tau|\leq r-1}\frac{|\tau|^{\alpha+2}}{(r-1)^4}d\tau\leq C'[g]_{0,\alpha}\frac{(r-1)^{\alpha+3}}{(r-1)^4}=C'[g]_{0,\alpha}(r-1)^{\alpha-1}.$$
Then for $r-1\leq |\tau|\leq \pi$:
$$ 2\left|\int_{r-1\leq |\tau|\leq \pi}(g(\tau+\phi)-g(\phi))\frac{1-\cos(\tau)}{((r-1)^2+2r(1-\cos(\tau)))^2}d\tau\right|$$
$$\leq [g]_{0,\alpha}C\int_{r-1\leq |\tau|\leq \pi}\frac{|\tau|^{\alpha+2}}{(2|\tau|^2)^2}d\tau\leq C'((r-1)^{\alpha-1}-\pi^{\alpha-1})\leq C'[g]_{0,\alpha}(r-1)^{\alpha-1}.$$
Finally, let us estimate the last integral for $|\tau|\leq r-1$:
$$ (r-1)^2\left|\int_{|\tau|\leq r-1}\frac{g(\tau+\phi)-g(\phi)}{((r-1)^2+2r(1-\cos(\tau)))^2}d\tau\right|$$
$$\leq [g]_{0,\alpha}C(r-1)^2\int_{|\tau|\leq r-1}\frac{|\tau|^{\alpha}}{(r-1)^4}d\tau\leq C'[g]_{0,\alpha}(r-1)^{\alpha-1}.$$

At last for $r-1\leq |\tau|\leq \pi$:

$$  (r-1)^2\left|\int_{r-1\leq |\tau|\leq \pi}\frac{g(\tau+\phi)-g(\phi)}{((r-1)^2+2r(1-\cos(\tau)))^2}d\tau\right|$$
$$ \leq  C[g]_{0,\alpha}(r-1)^2\int_{r-1\leq |\tau|\leq \pi}\frac{|\tau|^{\alpha}}{|\tau|^4}d\tau\leq C'[g]_{0,\alpha}(r-1)^2((r-1)^{\alpha-3}-\pi^{\alpha-3})$$
$$\leq C'[g]_{0,\alpha}(r-1)^{\alpha-1}.$$

In conclusion, we have:
$$|u(r_1e^{i\phi})-u(r_2e^{i\phi})|=\left|\int_{r_2}^{r_1}\frac{\partial u}{\partial r}dr\right|\leq \int_{r_2}^{r_1}\left|\frac{\partial u}{\partial r}\right|dr\leq C[g]_{0,\alpha}\int_{r_2}^{r_1}(r-1)^{\alpha-1}dr$$
$$\leq C'[g]_{0,\alpha}|r_1-r_2|^{\alpha},$$
and the result follows from the above.\\

\begin{proposition}   \label{prop7}

Let $g\in C_{per}^{0,\alpha}$, $u$ as in \eqref{PKer} $1<r_1\leq r_2\leq 2$, and $|\phi_1-\phi_2|\leq \pi$. Then:
$$ |u(r_1e^{i\phi_1})-u(r_2e^{i\phi_2})|\leq C[g]_{0,\alpha}|r_1e^{i\phi_1}-r_2e^{i\phi_2}|^{\alpha}.$$
(i.e. $[u]_{0,\alpha(B(0,2)\setminus B(0,1))}\leq C[g]_{0,\alpha (\partial B(0,1))}$).\\

\end {proposition}

\textbf{Proof}: Note that from the previous propositions we get:

$$ |u(r_1e^{i\phi_1})-u(r_2e^{i\phi_2})|\leq |u(r_1e^{i\phi_1})-u(r_1e^{i\phi_2})|+|u(r_1e^{i\phi_2})-u(r_2e^{i\phi_2})|$$
$$\leq C[g]_{0,\alpha (\partial B(0,1))}|r_1e^{i\phi_1}-r_1e^{i\phi_2}|^{\alpha}+C[g]_{0,\alpha (\partial B(0,1))}|r_1e^{i\phi_2}-r_2e^{i\phi_2}|^{\alpha}  $$
$$\leq C[g]_{0,\alpha (\partial B(0,1))}| r_1e^{i\phi_1}- r_2e^{i\phi_2}|^{\alpha}+C[g]_{0,\alpha (\partial B(0,1))}\left|r_2-r_1 \right|^{\alpha}$$
$$\leq C[g]_{0,\alpha (\partial B(0,1))}| r_1e^{i\phi_1}- r_2e^{i\phi_2}|^{\alpha},$$
because if $\theta$ is the angle between $r_1e^{i\phi_1}$ and $r_2e^{i\phi_2}$, we have:
$$|r_1e^{i\phi_1}-r_2e^{i\phi_2}|^2 -| r_1e^{i\phi_1}- r_1e^{i\phi_2}|^2=r_2^2-r_1^2-2r_1r_2\cos(\theta)+2r_1^2\cos(\theta)$$
$$=(r_2-r_1)(r_1+r_2-2r_1\cos(\theta))\geq (r_2-r_1)^2\geq 0.$$

\begin{proposition}   \label{prop8}

Let $g\in C_{per}^{1,\alpha}$, $u$ as in \eqref{PKer}, then:

$$\left\|\frac{\partial u}{\partial x_{\beta}}\right\|_{\infty (B(0,2)\setminus B(0,1))}\leq C\left\|g'\right\|_{0,\alpha(\partial B(0,1))} .$$

Moreover:
$$\left[\frac{\partial u}{\partial x_{\beta}}\right]_{0,\alpha (B(0,2)\setminus B(0,1))}\leq C\left\|g'\right\|_{0,\alpha(\partial B(0,1))}.  $$

\end {proposition}

\textbf{Proof}: 
Set $x=re^{i\phi}\in B(0,2)\setminus \overline{B(0,1)}$, $y=e^{i\tau}$.

Let $P(x;\tau)=\frac{1-|x|^2}{|x-y|^2}$, then: 
$$D_x(P(x;\tau))=D_x\left( \frac{1-|x|^2}{|x-y|^2} \right)=-2\left( \frac{x(|x-y|^2+1-|x|^2)-y(1-|x|^2)}{|x-y|^4}\right) .$$
Now, for $x\in B(0,2)\setminus \overline{B(0,1)}$, we have (due to the dominated convergence theorem):
$$D_x(u)=\frac{1}{2\pi}\int_{-\pi}^{\pi}D_x\left(P(x;\tau) \right)g(\tau)d\tau  .$$
In addition, the derivatives of $P$ are given by (note that we use $\tau=(\tau-\phi)+\phi$ and $|x-y|^2=1+r^2-2r\cos(\tau-\phi)$):
$$\frac{\partial P}{\partial x_1}=-2\frac{\cos(\phi)(2r-(r^2+1)\cos(\tau-\phi))+\sin(\phi)(1-r^2)\sin(\tau-\phi)}{(1+r^2-2r\cos(\tau-\phi))^2}$$

$$ \frac{\partial P}{\partial x_2}=-2\frac{\sin(\phi)(2r-(r^2+1)\cos(\tau-\phi))-\cos(\phi)(1-r^2)\sin(\tau-\phi)}{(1+r^2-2r\cos(\tau-\phi))^2}.$$
Furthermore:
$$\int_{-\pi}^{\pi}\frac{2r-(r^2+1)\cos(\tau-\phi)}{(1+r^2-2r\cos(\tau-\phi))^2}g(\tau)d\tau=-\int_{-\pi}^{\pi}\frac{d}{d\tau}\left(\frac{\sin(\tau-\phi)}{1+r^2-2r\cos(\tau-\phi)}\right)g(\tau)d\tau$$
$$=\int_{-\pi}^{\pi}\frac{\sin(\tau-\phi)}{1+r^2-2r\cos(\tau-\phi)}g'(\tau)d\tau=\int_{-\pi}^{\pi}\frac{\sin(\tau)}{1+r^2-2r\cos(\tau)}g'(\tau+\phi)d\tau.$$
Moreover:
$$ \int_{-\pi}^{\pi}\frac{(1-r^2)\sin(\tau-\phi)}{(1+r^2-2r\cos(\tau-\phi))^2}g(\tau)d\tau$$
$$=-\frac{1-r^2}{2r}\int_{-\pi}^{\pi}\frac{d}{d\tau}\left(\frac{1}{1+r^2-2r\cos(\tau-\phi)}\right)g(\tau)d\tau$$
$$=\frac{1}{2r}\int_{-\pi}^{\pi}\frac{1-r^2}{1+r^2-2r\cos(\tau-\phi)}g'(\tau)d\tau,$$

\noindent from the above, it is easy to conclude the result (using the estimates from the previous propositions and that $[\frac{\sin(\phi)}{r}]_{0,\alpha(B(0,2)\setminus B(0,1))}\leq C$, $[\frac{\cos(\phi)}{r}]_{0,\alpha(B(0,2)\setminus B(0,1))}\leq C$).

\begin{proposition}  \label{prop9}

Let $g\in C^{1,\alpha}(\partial B_1)$ and $u(x)=\int_{\partial B_1}g(y)\log|y-x|dS(y)$, then (for $1<|x|<2$) :\\
$\left\|Du\right\|_{\infty} \leq C (\left\|g\right\|_{\infty}+[g]_{0,\alpha}).$\\
$  [D u]_{0,\alpha}\leq C(\left\|g\right\|_{\infty}+[g]_{0,\alpha}). $\\
$ \left\|D^2 u\right\|_{\infty} \leq C(\left\|g\right\|_{\infty}+[g]_{0,\alpha}+\left\|g'\right\|_{\infty}+[g']_{0,\alpha}).  $\\
$  [D^2 u]_{0,\alpha}\leq C(\left\|g\right\|_{\infty}+[g]_{0,\alpha}+\left\|g'\right\|_{\infty}+[g']_{0,\alpha}). $\\

\end {proposition}

\textbf{Proof}:

The gradient of $u$ is given by:
$$Du(x)=\int_{-\pi}^{\pi}g(\tau)\frac{x-y}{|x-y|^2}d\tau,$$
with $y=(\cos(\tau),\sin(\tau))$ and $x=|x|e^{i\phi}$. Now, if $e_r(\tau)=(\cos(\tau),\sin(\tau))$ and \\
$e_{\tau}(\tau)=(-\sin(\tau), \cos(\tau))$, we have:

$$ Du(x)=\int_{-\pi}^{\pi}g(\tau)e_r(\tau)\frac{|x|\cos(\tau-\phi)-1}{|x-y|^2}d\tau-\int_{-\pi}^{\pi}g(\tau)e_{\tau}(\tau)\frac{|x|\sin(\tau-\phi)}{|x-y|^2}d\tau.$$
Note that $g_1:=g(\tau)e_r(\tau)$ and $g_2:=g(\tau)e_{\tau}(\tau)$ are $C^{1,\alpha}$ as functions of $\tau$.\\
If we call $v_1$ and $v_2$ to the first and second integral respectively, we get:
$$v_1(x)=\frac{-1}{2}\int_{-\pi}^{\pi}g_1(\tau)\left(1+\frac{1-|x|^2}{|x-y|^2}\right)d\tau.$$
On the other hand we have:
$$v_2=\frac{1}{2}\int_{-\pi}^{\pi}g_2(\tau)\frac{d}{d\tau}\left( \log(|x-y|^2)\right)d\tau=\frac{-1}{2}\int_{-\pi}^{\pi}\frac{d}{d\tau}g_2(\tau)\log\left(|x-y|^2\right)d\tau$$
$$ + \left.\frac{1}{2}g_2(\tau)\log\left(|x-y|^2 \right)\right|_{\tau=-\pi}^{\tau=\pi}=-\int_{-\pi}^{\pi}\frac{d}{d\tau}g_2(\tau)\log\left(|x-y|\right)d\tau.$$
If we repeat the argument (to each component) we get:

$$ D(v_2^{(j)})=\frac{1}{2}\int_{-\pi}^{\pi}g_2'^{(j)}(\tau)e_r(\tau)\left(1+\frac{1-|x|^2}{|x-y|^2}\right)d\tau+\int_{-\pi}^{\pi}g_2'^{(j)}(\tau)e_{\tau}(\tau)\frac{|x|\sin(\tau-\phi)}{|x-y|^2}d\tau.$$

It is easy to see (using the estimates from the previous propositions) that:
$$|Du|\leq C (\left\|g\right\|_{\infty}+[g]_{0,\alpha}).$$

Moreover:

$$ |D^2 u|\leq C(\left\|g\right\|_{\infty}+[g]_{0,\alpha}+\left\|g'\right\|_{\infty}+[g']_{0,\alpha}).  $$

Furthermore:
$$  [D^2 u]_{0,\alpha}\leq C(\left\|g\right\|_{\infty}+[g]_{0,\alpha}+\left\|g'\right\|_{\infty}+[g']_{0,\alpha}). $$

(It may be useful to know the following estimates, where $\beta$ represents either $r$ or $\tau$ :

$[g_k']_{0,\alpha}\leq C(\left\|g\right\|_{\infty}+[g]_{0,\alpha}+\left\|g'\right\|_{\infty}+[g']_{0,\alpha}).$\\

$ [g_k'^{(j)}e_\beta]_{0,\alpha}\leq C (\left\|g\right\|_{\infty}+[g]_{0,\alpha}+\left\|g'\right\|_{\infty}+[g']_{0,\alpha}).$\\

$[g_k]_{0,\alpha}\leq C(\left\|g\right\|_{\infty}+[g]_{0,\alpha}).$\\

$[e_\beta]_{0,\alpha}\leq C$).

\begin{proposition}  \label{prop10}

Let $g\in C^{1,\alpha}(\partial B_R)$ and $u=\int_{\partial B_R}g\log|y-x|dS$, then (for $R<|x|<R+d$, with $d\leq R$) :\\
$\left\|Du\right\|_{\infty} \leq C (\left\|g\right\|_{\infty}+R^{\alpha}[g]_{0,\alpha}).$\\
$  [D u]_{0,\alpha}\leq C(R^{-\alpha}\left\|g\right\|_{\infty}+[g]_{0,\alpha}). $\\
$ \left\|D^2 u\right\|_{\infty} \leq C(R^{-1}\left\|g\right\|_{\infty}+R^{\alpha-1}[g]_{0,\alpha}+\left\|g'\right\|_{\infty}+R^{\alpha}[g']_{0,\alpha}).  $\\
$  [D^2 u]_{0,\alpha}\leq C(R^{-1-\alpha}\left\|g\right\|_{\infty}+R^{-1}[g]_{0,\alpha}+R^{-\alpha}\left\|g'\right\|_{\infty}+[g']_{0,\alpha}). $\\

\end {proposition}

\textbf{Proof}: It follows by a rescaling argument.

\begin{proposition}   \label{prop11}

Let $u=\int_{\partial B_R}gG_N(x,y)dS(y)$, then:\\
$\left\|Du\right\|_{\infty(B(0,R+d)\setminus\overline{B(0,R)})} \leq C (\left\|g\right\|_{\infty}+R^{\alpha}[g]_{0,\alpha}).$\\
$  [D u]_{0,\alpha(B(0,R+d)\setminus\overline{B(0,R)})}\leq C(R^{-\alpha}\left\|g\right\|_{\infty}+[g]_{0,\alpha}) .$\\
$ \left\|D^2 u\right\|_{\infty(B(0,R+d)\setminus\overline{B(0,R)})} \leq C(R^{-1}\left\|g\right\|_{\infty}+R^{\alpha-1}[g]_{0,\alpha}+\left\|g'\right\|_{\infty}+R^{\alpha}[g']_{0,\alpha}) . $\\
$  [D^2 u]_{0,\alpha(B(0,R+d)\setminus\overline{B(0,R)})}\leq C(R^{-1-\alpha}\left\|g\right\|_{\infty}+R^{-1}[g]_{0,\alpha}+R^{-\alpha}\left\|g'\right\|_{\infty}+[g']_{0,\alpha}) .$\\

\end {proposition}

\textbf{Proof}: Thanks to \eqref{log-reflection} we have: 
$$G_N(x,y)=-\frac{1}{\pi}\log|y-x|+\frac{1}{2\pi}\log\frac{|x|}{R}-\frac{|y|^2}{4\pi R^2}. $$

The estimates for $u$ then follow from Proposition \ref{prop10} and estimates for $\log|x|$ (recall that for the H\"older continuity, we can proceed as in Proposition \ref{prop2}).\\

\begin{lemma} \label{trace}
Let $\phi\in H^{1}(B_{\rho_2}\setminus \overline{B_{\rho_1}})$ for some $0<\rho_1<\rho_2$. Then (for $i=1,2$):
$$\int_{\partial B_{\rho_i}}\phi^2(x)dS(x)\leq 2\frac{\rho_2}{\rho_1}\left(  
\frac{1}{\rho_2-\rho_1}\int_{B_{\rho_2}\setminus \overline{B_{\rho_1}}}\phi^2(x)dx+\int_{B_{\rho_2}\setminus \overline{B_{\rho_1}}}|D\phi|^2(x)dx  \right)$$

\end{lemma}

\noindent\textbf{Proof}: We consider only the case of $\int_{\partial B_{\rho_2}}\phi^2$ (the other case is analogous). Given $\varepsilon >0$, let $\eta\in C^{\infty}(\overline{B_{\rho_2}}\setminus \overline{B_{\rho_1}})$ be such that $\eta= 1$ on $\partial B_{\rho_2}$, $\eta=0$ on $\partial B_{\rho_1}$ and $|D\eta|\leq \frac{1+\varepsilon}{\rho_2-\rho_1}$.
$$    \int_{\partial B_{\rho_2}}\phi^2(x)dS(x)=\rho_2\int_{S^{1}}\left(\int_{\rho_1}^{\rho_2}\frac{d}{ds}((\eta \phi)(sz))ds\right)^2dS(z)     $$
$$\leq 2\rho_2(\rho_2-\rho_1)\int_{S^1}\int_{\rho_1}^{\rho_2}(|\phi D \eta|^2+|\eta D\phi|^2)dsdS(z)    $$
$$\leq 2\frac{\rho_2}{\rho_1}\left(     \frac{1+\varepsilon}{\rho_2-\rho_1}\int_{\rho_1}^{\rho_2}\int_{S^{1}}\phi^2(x)\frac{\rho_1}{s}dS(x)ds +\int_{\rho_1}^{\rho_2}\int_{S^{1}}|D\phi|^2(x)\frac{\rho_1}{2}dS(x)ds     \right).$$

\begin{proposition} \label{prop12}
Let $E$ and $d$ be as in \eqref{eq:genericE} and \eqref{eq:d}.
Let $u$ be such that: 
\[ 
\begin{cases}
\Delta u =0, &  in \hspace{0.2cm} E\\
\frac{\partial u}{\partial \nu}=g, &  on \hspace{0.2cm} \partial E
\end{cases}
\]
and $\int_{E}u(y)dy=0$. Set 
\begin{align}\label{eq:B}
B=B(E):= |E|^{\frac{1}{2}}C_P(E)\Big (d^{-\frac{1}{2}}C_P(E)+1\Big )n^{\frac{1}{2}}r_0^{\frac{1}{2}}. 
\end{align}
Then: $$ \Vert u \Vert_{L^{1}(E)}\leq C\cdot B\Vert g\Vert_{\infty}.   $$

\end {proposition}

\noindent\textbf{Proof}: First note that:
$$ \int_{E}|u|dy\leq |E|^{\frac{1}{2}}\Vert u\Vert_{L^{2}(E)}\leq C_P|E|^{\frac{1}{2}}\Vert Du\Vert_{L^{2}(E)}.$$

\noindent Now, using integration by parts we get:

$$\int_{E}u\Delta u dy=\int_{\partial E}ugdS(y)-\int_{ E}|Du|^{2}dy=0.$$
Moreover:

$$ \int_{E}|Du|^2dy\leq \Vert g\Vert_{L^2(\partial E)}\Vert u\Vert_{L^2(\partial E)}.$$
Using Cauchy's inequality, we get:

$$  \Vert Du\Vert_{L^{2}(E)}\leq \frac{1}{2^{\frac{1}{2}}}\left( A\Vert g\Vert_{L^{2}(\partial E)} +\frac{\Vert u\Vert_{L^{2}(\partial E)}}{A} \right),  $$

\noindent furthermore, using \ref{trace} and Poincare constant, we obtain:

$$ \int_{\partial E}u^2dS=\sum_{k=0}^{n}\int_{\partial B(z_k,r_k)}u^2dS\leq C\left(\int_{B(z_0,r_0)\setminus B(z_0,r_0-d)}d^{-1}u^2+|Du|^2dy\right)                  $$
$$ +C\left(  \sum_{k=1}^{n}\int_{B(z_k,r_k+d)\setminus B(z_k,r_k)}d^{-1}u^2+|Du|^2dy       \right)$$
$$\leq C\left( d^{-1}\int_{E}u^2dy+\int_{E}|Du|^2dy    \right)\leq C(d^{-1}C_P^2+1)\int_{E}|Du|^2dy$$

\noindent Choosing $A=2^{\frac{1}{2}}C(d^{\frac{-1}{2}}C_P+1)$ we deduce that:

$$  \Vert Du\Vert_{L^{2}(E)}\leq 2^{\frac{1}{2}}A\Vert g\Vert_{L^2(\partial E)}\leq C(d^{\frac{-1}{2}}C_P+1)n^{\frac{1}{2}}r_0^{\frac{1}{2}}\Vert g\Vert_{\infty}.   $$

\noindent Finally, we obtain:

$$ \Vert u \Vert_{L^{1}(E)}\leq C\cdot |E|^{\frac{1}{2}}C_P(d^{\frac{-1}{2}}C_P+1)n^{\frac{1}{2}}r_0^{\frac{1}{2}}\Vert g\Vert_{\infty}.   $$

\begin{proposition} \label{prop13}
(regularity near the holes)
Let $B$ and $u$ be as in Proposition \ref{prop12}, then, if $A=\cup_{k=1}^{n}B(z_k,r_k+\frac{d}{3})\setminus\overline{B(z_k,r_k)}$, we have:\\
$\Vert Du \Vert_{L^{\infty}(A)}\leq C\left(1+Bd^{-4}r_0\right)\Vert g \Vert_{\infty}+Cr_{0}^{\alpha}[g]_{0,\alpha}.$ \\
$[Du]_{0,\alpha(B(z_k,r_k+\frac{d}{3})\setminus\overline{B(z_k,r_k)})}\leq C\left(Bd^{-5}r_0^{2-\alpha}+d^{-\alpha}\right)\Vert g \Vert_{\infty}+C[g]_{0,\alpha}.$  \\
$\Vert D^2u \Vert_{L^{\infty}(A)}\leq C\left(Bd^{-5}r_0+d^{-1}\right)\Vert g \Vert_{\infty}+Cd^{\alpha-1}[g]_{0,\alpha}+C\Vert g' \Vert_{\infty}+Cr_{0}^{\alpha}[g']_{0,\alpha}.$\\
$[D^2u]_{0,\alpha(B(z_k,r_k+\frac{d}{3})\setminus\overline{B(z_k,r_k)})}\leq C\left(Bd^{-6}r_{0}^{2-\alpha}+d^{-1-\alpha}\right)\Vert g \Vert_{\infty}+Cd^{-1}[g]_{0,\alpha}+Cd^{-\alpha}\Vert g' \Vert_{\infty}+C[g']_{0,\alpha}.$\\

\end {proposition}

\noindent \textbf{Proof}: It follows from Proposition \ref{prop1}, Proposition \ref{prop11}, Proposition \ref{prop4}, Lemma \ref{lemma1} and Proposition\ref{prop12} (recall that $r_i\geq d$).

\begin{proposition} \label{prop14}

(interior regularity) Let $B$ as in proposition \ref{prop12}, $u$ be harmonic in $E$ and $E'=B(z_0,r_0-\frac{d}{3})\setminus \bigcup_{k=1}^{n}B(z_k,r_k+\frac{d}{3})$, then:\\

\noindent$ \left\|u\right\|_{L^{\infty}(E')}\leq C d^{-2}\left\|u\right\|_{L^1(E)}\leq  C Bd^{-2} \left\|g\right\|_{\infty}.$\\
$ [u  ]_{0,\alpha(E')}\leq Cd^{-3}r_{0}^{1-\alpha}\left\|u\right\|_{L^1(E)}\leq CBd^{-3}r_{0}^{1-\alpha}\left\|g\right\|_{\infty}.$\\
$ \left\|D^{\beta}u\right\|_{L^{\infty}(E')}\leq C d^{-2-|\beta|}\left\|u\right\|_{L^1(E)}\leq  C Bd^{-2-|\beta|}\left\|g\right\|_{\infty} .$\\
$ [ u ]_{1,\alpha(E')}\leq Cd^{-4}r_{0}^{1-\alpha}\left\|u\right\|_{L^1(E)}\leq CBd^{-4}r_{0}^{1-\alpha}\left\|g\right\|_{\infty}.$\\
$ [ D^{2}u ]_{0,\alpha(E')}\leq Cd^{-5}r_{0}^{1-\alpha}\left\|u\right\|_{L^1(E)}\leq CBd^{-5}r_{0}^{1-\alpha}\left\|g\right\|_{\infty}.$

\end {proposition}

\noindent \textbf{Proof}: It follows from local regularity for harmonic functions and Proposition \ref{prop2} (using triangle inequality at most $2n+1$ times): Join $x$ and $z$ with a straight line, then the segment intersects at most the $n$ holes. In that case, join the points using segments of the above straight line and segments of circles of the form $\partial B(z_k,r_k+\frac{d}{3})$ (for straight lines use local estimates for harmonic functions and for circles use Proposition \ref{prop2}).

\begin{proposition} \label{prop15}

Let $v$ be harmonic in $\Omega$ and $\zeta$ be a cut-off function equal to $0$ for $|x|\leq R+\frac{d}{3}$ and equal to $1$ for $R+\frac{2}{3}d\leq |x|$, then, if $u=\zeta v$:

$$u(x)=C+\int_{\partial B_R}\frac{\partial u}{\partial \nu}\left( \Phi(y-x)-\phi^{x}(y) \right)dS(y)-\int_{\Omega}\Delta u \left( \Phi(y-x)-\phi^{x}(y)  \right)dy.$$

\end {proposition}

\noindent \textbf{Proof}: This can be showed using the same techniques as in the proof of Proposition \ref{prop1}. 

\noindent The proofs of the following two results, are similar to the proof of Lemma \ref{lemma1} and Proposition \ref{prop11} respectively :

\begin{lemma} \label{lemma16}

Let $R\geq Cd$, $v$ be harmonic in $\Omega$ and $\zeta$ be a cut-off function equal to $0$ for $|x|\leq R+\frac{d}{3}$ and equal to $1$ for $R+\frac{2}{3}d\leq |x|$, then:\\
$ [\Delta(v\zeta)]_{0,\alpha(\mathbb{R}^2)}\leq CR^{1-\alpha}d^{-5}\left\|v\right\|_{L^1(\Omega)}. $\\
$ \left\|\Delta(v\zeta)\right\|_{\infty (\mathbb{R}^2)}\leq Cd^{-4}  \left\|v\right\|_{L^1(\Omega)} .$

\end {lemma}

\begin{proposition}   \label{prop16}
Let $u=\int_{\partial B_{r_0}}gG_N(x,y)dS(y)$, then:\\
$\left\|Du\right\|_{\infty(B(0,r_0)\setminus\overline{B(0,r_0-\frac{d}{3})})} \leq C (\left\|g\right\|_{\infty}+r_{0}^{\alpha}[g]_{0,\alpha}).$\\
$  [D u]_{0,\alpha(B(0,r_0)\setminus\overline{B(0,r_0-\frac{d}{3})})}\leq C(r_{0}^{-\alpha}\left\|g\right\|_{\infty}+[g]_{0,\alpha}) .$\\
$ \left\|D^2 u\right\|_{\infty(B(0,r_0)\setminus\overline{B(0,r_{0}-\frac{d}{3})})} \leq C(r_{0}^{-1}\left\|g\right\|_{\infty}+r_{0}^{\alpha-1}[g]_{0,\alpha}+\left\|g'\right\|_{\infty}+r_{0}^{\alpha}[g']_{0,\alpha}) . $\\
$  [D^2 u]_{0,\alpha(B(0,r_0)\setminus\overline{B(0,r_{0}-\frac{d}{3})})}\leq C(r_{0}^{-1-\alpha}\left\|g\right\|_{\infty}+r_{0}^{-1}[g]_{0,\alpha}+r_{0}^{-\alpha}\left\|g'\right\|_{\infty}+[g']_{0,\alpha}) .$\\

\end {proposition}

\begin{proposition} \label{prop17}
(regularity near the exterior boundary)
Let $B$ and $u$ be as in Proposition \ref{prop12}, then, we have:\\
$\left\|Du\right\|_{\infty(B(0,r_0)\setminus\overline{B(0,r_0-\frac{d}{3})})} \leq C (1+Bd^{-4}r_0)\left\|g\right\|_{\infty}+Cr_{0}^{\alpha}[g]_{0,\alpha}.$\\
$  [D u]_{0,\alpha(B(0,r_0)\setminus\overline{B(0,r_0-\frac{d}{3})})}\leq C(r_{0}^{-\alpha}+Bd^{-5}r_{0}^{2-\alpha})\left\|g\right\|_{\infty}+C[g]_{0,\alpha} .$\\
$ \left\|D^2 u\right\|_{\infty(B(0,r_0)\setminus\overline{B(0,r_{0}-\frac{d}{3})})} \leq C(r_{0}^{-1}+Bd^{-5}r_0)\left\|g\right\|_{\infty}+Cr_{0}^{\alpha-1}[g]_{0,\alpha}+C\left\|g'\right\|_{\infty}+Cr_{0}^{\alpha}[g']_{0,\alpha} . $\\
$  [D^2 u]_{0,\alpha(B(0,r_0)\setminus\overline{B(0,r_{0}-\frac{d}{3})})}\leq C(r_{0}^{-1-\alpha}+Bd^{-6}r_{0}^{2-\alpha})\left\|g\right\|_{\infty}+Cr_{0}^{-1}[g]_{0,\alpha}+Cr_{0}^{-\alpha}\left\|g'\right\|_{\infty}+C[g']_{0,\alpha} .$\\

\end {proposition}

\noindent \textbf{Proof}: It follows from Proposition \ref{prop15}, Proposition \ref{prop16}, Proposition \ref{prop4}, Lemma \ref{lemma16} and Proposition \ref{prop12} (recall that $r_0\geq Cd$).

\begin{theorem} \label{thm1}
(global regularity)
Let $B$ and $u$ be as in Proposition \ref{prop12}, then, we have:\\
$\left\|Du\right\|_{\infty(E)} \leq C (1+Bd^{-4}r_0)\left\|g\right\|_{\infty}+Cr_{0}^{\alpha}[g]_{0,\alpha}.$\\
$  [D u]_{0,\alpha(E)}\leq C(d^{-\alpha}+Bd^{-5}r_{0}^{2-\alpha})\left\|g\right\|_{\infty}+C[g]_{0,\alpha} .$\\
$ \left\|D^2 u\right\|_{\infty(E)} \leq C(d^{-1}+Bd^{-5}r_0)\left\|g\right\|_{\infty}+Cd^{\alpha-1}[g]_{0,\alpha}+C\left\|g'\right\|_{\infty}+Cr_{0}^{\alpha}[g']_{0,\alpha} . $\\
$  [D^2 u]_{0,\alpha(E)}\leq C(d^{-1-\alpha}+Bd^{-6}r_{0}^{2-\alpha})\left\|g\right\|_{\infty}+Cd^{-1}[g]_{0,\alpha}+Cd^{-\alpha}\left\|g'\right\|_{\infty}+C[g']_{0,\alpha} .$\\

\end {theorem}

\noindent \textbf{Proof}: It follows from Proposition \ref{prop13}, Proposition \ref{prop14} and Proposition \ref{prop17} (recall that $r_0\geq Cd$).

\noindent

\begin{theorem}   \label {thm 2}
Let $0<\delta <1$. There exists a universal constant $C(\delta)$ such that $C_P(E)\leq C(\delta)r_0$
for every $E=B(z_0,r_0)\setminus\bigcup_{i=1}^{n}B_i \in \mathcal{F}_{\delta}$.
\end{theorem}

\begin{proof}
By a simple rescaling argument, it is enough to consider the case when $r_0=1$ and $z_0=0$. 
Using cut-off functions and elementary reflections we may define an extension 
operator $\Psi_E:H^{1}\left(B(0,1)\setminus\bigcup_{i=1}^{n}B_i\right)\rightarrow H^{1}(B(0,1))$ such that:\\
$\Vert \Psi_E\phi \Vert_{L^{2}(B(0,1))}\leq C\Vert \phi \Vert_{L^{2}(E)}$, 
$\Vert D(\Psi_E\phi) \Vert_{L^{2}(B(0,1))}\leq C( \delta^{-1}\Vert \phi \Vert_{L^{2}(E)}
+\Vert D\phi \Vert_{L^{2}(E)})$ (the constants can be chosen as $2$ and $4$ respectively).\\
To prove this, assume, for a contradiction that:
$$\int_{E_j}\phi_{j}=0,\medspace \int_{E_j}\phi_{j}^{2}=1,\medspace \text{ and}\medspace  \int_{E_j}|D\phi_j|^2<\frac{1}{j}, $$
for some sequence of sets $E_j=B(0,1)\setminus\bigcup_{i=1}^{n}B_{i}^{(j)}\in \mathcal{F}_{\delta}$ and
maps $\phi_j\in H^{1}(E_j)$. Call $\tilde{\phi}=\Psi_j\phi_j$, $\Psi_j$ being the extension operator for $E_j$. Taking subsequences we find $E=B(0,1)\setminus \bigcup_{i=1}^{n}B_i\in \mathcal{F}_{\delta}$ and $\phi\in H^{1}(B(0,1))$ such that:
$$1\leq \int_{B(0,1)}\tilde{\phi}^2<C, \text{\medspace}  \int_{B(0,1)}|D\tilde{\phi}|^2<C(\delta^{-2}+\frac{1}{j})   ,\text{\medspace}  \tilde{\phi}\rightharpoonup \phi\text{ in\medspace}  H^{1}(B(0,1)) , \text{\medspace}    |E_j\Delta E|\rightarrow 0.$$
Also, for every $E'=B(0,1)\setminus \bigcup_{i=1}^{n}B_i^{'}\in \mathcal{F}_{\delta} $ such that$E\subset\subset E'$ we have 
$D\tilde{\phi_j}=D\phi_j\rightarrow 0$ in $L^{2}(E')$. By uniqueness of the weak limit, $D\phi=0$ in every such $E'$, hence $\phi$ is constant in $E$. By the compact embedding of $H^{1}(B(0,1))$ into $L^{2}(B(0,1))$ we can assume that $\tilde{\phi}\rightharpoonup \phi$ in $L^2(B(0,1))$, so:
$$0=\lim\int_{E_j}\phi_j=\lim\int_{E_j}\tilde{\phi_j}\chi_{E}=\int_{B(0,1)}\phi\chi_{E}.$$ 
Thus $\phi=0$ in $E$. However, by the compact embedding the convergence is not only in $L^{2}(B(0,1))$, we can take a higher exponent, so also:
$$1=\lim\int_{E_j}\phi_j^2=\lim\int_{E_j}\tilde{\phi_j}^2\chi_{E}=\int_{B(0,1)}\phi^2\chi_{E},$$
which gives a contradiction.
\end{proof}

Given $z_1,...,z_n\in \mathbb{R}^2$ and $d,r_0,...r_n>0$ satisfying that
\begin{align}\label{property} 
\begin{aligned}
    & \text{$\bullet$\ $r_i\geq d$ for each $i\in \{1,\ldots, n\}$; and}\\ 
    & \text{$\bullet$\ the  $\textstyle \overline{B(z_i,r_i+d)}$ are disjoint and contained 
    in $\textstyle  B(z_0,r_0-d)$;}
\end{aligned}
\end{align}

we consider the boundary value problem

\begin{align} \label{bvp}
\left\{ \begin{aligned}
div \medspace v&=0& &\text{in $ E$,} \\
v(x) &=g(x)\nu(x) & &\text{on $ \partial E$ ,}
\end{aligned} \right. 
\end {align}

where \begin{equation}
g\in C^{1,\alpha}\left( \bigcup_{i=0}^{n}\partial B_i\right) 
\text{ and } \int_{\partial B_0}g=\sum_{i=1}^{n}\int_{\partial B_i}g \label{nec}
\end{equation}
(with $B_i:=B(z_i,r_i)$ and $B_0:=B(z_0, r_0)$).

\begin{theorem}\label{thm3}
Let $n\in\mathbb{N}$, $0<\delta<1$ and $B$ as in Proposition \ref{prop12}. 
There exist a universal constant $C_3$ such that whenever $z_1,...z_n\in\mathbb{R}^2$ and $d,r_0,...,r_n>0$
satisfy $\frac{d}{r_0}\geq \delta$ and \eqref{property}, we have that for every $g$ verifying \eqref{nec} it is possible to construct a solution to \eqref{bvp} for which
$$\Vert v\Vert_{\infty}\leq C_3\left(\left(\left(\frac{r_0}{d}\right)^{1+\alpha}+B\left(\frac{r_0}{d^2}\right)^{3}+B^2\left(\frac{r_0}{d^3}\right)^{3}\right)\Vert g\Vert_{\infty}+\left(\frac{r_0^{2\alpha+1}}{d^{\alpha+1}}+B\frac{r_0^{2+\alpha}}{d^5}\right)[g]_{0,\alpha}\right). $$

$$\Vert Dv\Vert_{\infty}\leq C_3\left( C_1\Vert g\Vert_{\infty}+C_2[g]_{0,\alpha}+\frac{r_0^{\alpha}}{d^{\alpha}}\Vert g'\Vert_{\infty}+\frac{r_0^{2\alpha}}{d^{\alpha}}[g']_{0,\alpha}\right),$$
where $C_1=r_0^{1+\alpha}d^{-\alpha-2}+Br_0^{3}d^{-7}+B^2r_0^{3}d^{-10}$ and $C_2=r_0^{2\alpha+1}d^{-2-\alpha}+Br_0^{2+\alpha}d^{-6}$

\end{theorem}

\noindent \textbf{Proof}: To prove this we follow the strategy of Dacorogna-Moser \cite{DaMo90} which 
consists in solving first

\begin{align} \label{phi bvp}
\left\{ \begin{aligned}
\Delta \phi&=0& &\text{in $ E$,} \\
\frac{\partial \phi}{\partial \nu} &=g(x) & &\text{on $ \partial E$ ,}
\end{aligned} \right. 
\end {align}
with $\int_{E}\phi=0$ and then choosing $v=D\phi+D^{\perp}\psi$ where $D^{\perp}\psi:=(\partial_{z_2}\psi,-\partial_{z_1}\psi)$ is a divergence-free covector field that cancels out the tangential parts of $D\phi$ on $\partial B_i, \forall i$. Concretely
$\psi(z)=\varphi(z)-\zeta\left(\frac{2dist(z,\partial E)}{d}\right)\varphi(q(z))$ where $\varphi$ is the solution to

\begin{align} \label{varphi bvp}
\left\{ \begin{aligned}
\Delta \varphi&=0& &\text{in $ E$,} \\
\frac{\partial \varphi}{\partial \nu} &=\frac{\partial \phi}{\partial \tau} & &\text{on $ \partial E$ ,}
\end{aligned} \right. 
\end {align}

\begin{equation} \label{boundary proj}
q(z) =
\left\{
	\begin{array}{ll}
		r_k\frac{z-z_k}{|z-z_k|}+z_k  & \mbox{if } |z-z_k| <r_k+\frac{d}{2} \\
		r_0\frac{z}{|z|} & \mbox{if } |z| > r_0-\frac{d}{2}
	\end{array}
\right.
\end{equation}

and $\zeta$ is a cutoff function such that $0\leq \zeta  \leq 1$, $\zeta(0)=1$ and $\zeta(1)=0$.\\

Using Theorem \ref{thm1} we get the following estimates:

$$ \Vert D\varphi \Vert_{\infty}\leq C\left( (1+Bd^{-4}r_0)\left\Vert \frac{\partial \phi}{\partial \tau} \right\Vert_{\infty} +r_0^{\alpha}\left\Vert \frac{\partial \phi}{\partial \tau} \right\Vert_{0,\alpha} \right)$$

$$ \Vert D^2\varphi \Vert_{\infty}\leq C\left( (d^{-1}+Bd^{-5}r_0)\left\Vert \frac{\partial \phi}{\partial \tau} \right\Vert_{\infty} +d^{\alpha-1}\left[ \frac{\partial \phi}{\partial \tau} \right]_{0,\alpha} +\left\Vert\frac{\partial^2\phi}{\partial\tau^2}  \right\Vert_{\infty}+r_0^{\alpha}\left[ \frac{\partial^2\phi}{\partial\tau^2} \right]_{0,\alpha}\right).$$

Now, it is easy to see that:

$$   \left\Vert \frac{\partial \phi}{\partial \tau} \right\Vert_{\infty}\leq  \Vert D\phi \Vert_{\infty} $$
$$   \left[ \frac{\partial \phi}{\partial \tau} \right]_{0,\alpha} \leq C\left(   d^{-\alpha}\Vert D\phi \Vert_{\infty}+[D\phi]_{0,\alpha}  \right)$$
$$   \left\Vert \frac{\partial^2 \phi}{\partial \tau^2} \right\Vert_{\infty} \leq C\left( d^{-1}\Vert D\phi \Vert_{\infty} +\Vert D^2\phi\Vert_{\infty}   \right)$$
$$   \left[ \frac{\partial^2 \phi}{\partial \tau^2} \right]_{0,\alpha}\leq C\left(   d^{-1-\alpha}\Vert D\phi\Vert_{\infty}+d^{-1}\left[D\phi\right]_{0,\alpha}+d^{-\alpha}\Vert D^2\phi \Vert_{\infty}+\left[D^2\phi\right]_{0,\alpha} \right). $$
Moreover:

$$   \left\Vert \frac{\partial \phi}{\partial \tau} \right\Vert_{\infty}\leq C\left( (1+Bd^{-4}r_0)\Vert g \Vert_{\infty} +r_0^{\alpha}[g]_{0,\alpha}\right) $$
$$   \left[ \frac{\partial \phi}{\partial \tau} \right]_{0,\alpha} \leq C\left(   (d^{-\alpha}+Bd^{-5}r_0^{2-\alpha})\Vert g \Vert_{\infty}+\frac{r_0^{\alpha}}{d^{\alpha}}[g]_{0,\alpha}  \right)$$
$$   \left\Vert \frac{\partial^2 \phi}{\partial \tau^2} \right\Vert_{\infty} \leq C\left( (d^{-1}+Bd^{-5}r_0)\Vert g \Vert_{\infty} +r_{0}^{\alpha}d^{-1}[g]_{0,\alpha}+\Vert g' \Vert_{\infty} +r_0^{\alpha}[g']_{0,\alpha}   \right)$$
$$   \left[ \frac{\partial^2 \phi}{\partial \tau^2} \right]_{0,\alpha}\leq C\left(   (d^{-1-\alpha}+Bd^{-6}r_0^{2-\alpha})\Vert g \Vert_{\infty} +d^{-1}\frac{r_0^{\alpha}}{d^{\alpha}}[g]_{0,\alpha}+d^{-\alpha}\Vert g' \Vert_{\infty} +\frac{r_0^{\alpha}}{d^{\alpha}}[g']_{0,\alpha} \right). $$

From the above we deduce that:

$$ \Vert D\varphi\Vert_{\infty}\leq C\left((r_0^{\alpha}d^{-\alpha}+Bd^{-5}r_0^{2}+B^2d^{-8}r_0^{2})\Vert g\Vert_{\infty}+\left( \frac{r_0^{2\alpha}}{d^{\alpha}}+Bd^{-4}r_0^{1+\alpha}\right)[g]_{0,\alpha}\right) $$
$$   \Vert D^2\varphi\Vert_{\infty}\leq C\left(A_1\Vert g\Vert_{\infty}+A_2[g]_{0,\alpha}+\frac{r_0^{\alpha}}{d^{\alpha}}\Vert g'\Vert_{\infty}+\frac{r_{0}^{2\alpha}}{d^{\alpha}}[g']_{0,\alpha}\right) ,$$
where $A_1=\left(\frac{r_0}{d}\right)^{\alpha}d^{-1}+Bd^{-6}r_0^{2}+B^2d^{-9}r_0^{2}$ and $A_2=\frac{r_0^{2\alpha}}{d^{1+\alpha}}+Bd^{-5}r_0^{1+\alpha}$

On the other hand, it is easy to see that:

$$\Vert D\psi\Vert_{\infty}\leq C\left(\frac{1}{d}\Vert\varphi\Vert_{\infty}+\Vert D\varphi\Vert_{\infty}\right)$$

$$\Vert D^2\psi\Vert_{\infty}\leq C\left(\frac{1}{d^2}\Vert\varphi\Vert_{\infty}+\frac{1}{d}\Vert D\varphi\Vert_{\infty}+\Vert D^2\varphi\Vert_{\infty}\right).$$

Note that using the fundamental theorem of calculus one can obtain (using that there exists a point where $\varphi$ vanishes): $\Vert \varphi \Vert_{\infty}\leq Cr_0\Vert D\varphi\Vert_{\infty}$. Finally the result follows by adding the estimates for $\varphi$.

\section{Proof of the main theorem}
\label{se:proof}

Let $n\in \N$, $R_0>0$, and ${\mathcal{B}}:=B(0,R_0)\subset \R^2$. 
Suppose that $\Big ( (a_i)_{i=1}^n, (v_i)_{i=1}^n \Big )$ is an attainable configuration.
Let $\lambda>1$, $z_i:[1,\lambda]\to \R^2$ and $L_i:[1,\lambda] \to [0,\infty)$, $i\in \{1,\ldots, n\}$,
be as in Definition \ref{de:attainable}.
By continuity, there exist $R_1, \ldots, R_n >0$  such that for
\begin{align} \label{eq:defRiri}
 r_i(t):= \sqrt{ L_i(t)^2 + R_i^2}, \quad t\in [1,\lambda],\quad i\in \{1,\ldots, n\}
\end{align}
the balls $\overline{B(z_i(t), r_i(t))}$ are disjoint and contained in $B(0, r_0(t))$, with
$$r_0(t):=tR_0,$$
for every $t\in [1,\lambda]$.

Most of the conclusions of Theorem \ref{th:main} are obtained exactly as in \cite[Thm.\ 1.9]{HeSe13}.
The novelty in this work is to solve the nonlinear equation of incompressibility 
for an arbitrarily large number of cavities. Near each cavitation point (to be precise, in 
$\{x:\ \epsilon\leq |x-a_i|\leq R_i\}$), we work with the unique radially symmetric deformations
creating cavities of the desired sizes.

\begin{proposition} \label{pr:scission-near}
 Let $u:\mathcal{B}\to \R^2$ be such that for every $i$ and $0<r<R_i$
 $$ 
 u(a_i+re^{i\theta}) =
  z_i(\lambda) + \sqrt{L_i(\lambda)^2 + r^2} e^{i\theta}.
$$
Then $u|_{\bigcup B(a_i, R_i)}$ is one-to-one a.e., satisfies $\det Du\equiv 1$ a.e.,
and is such that $|\imT(u, B_\eps(a_i))|=v_i+\pi\eps^2$ for all $i$ and
$$
    \int_{\bigcup \{x:\eps<|x-a_i|<R_i\}} \frac{|Du|^2}{2}\dd x \leq 
    \sum_i \pi R_i^2 + \sum_i v_i\log R_i +  \left (\sum_{i=1}^n v_i\right )|\log \eps|
$$
for every small $\eps>0$.
\end{proposition} \label{pr:near}
\begin{proof}
  Given $i\in\{1,\ldots, n\}$, $r\in (0,R_i)$ and $\theta\in[0,2\pi]$
  \begin{align}
   Du(a_i+re^{i\theta}) &= \frac{r}{\sqrt{L_i(\lambda)^2 +r^2}} e^{i\theta}\otimes e^{i\theta}
      + \sqrt{1+\frac{L_i(\lambda)^2}{r^2}} ie^{i\theta}\otimes ie^{i\theta}.
  \end{align}
  Hence $\det Du\equiv 1$ and 
  \begin{align}
   \int_{\bigcup \{x:\eps<|x-a_i|<R_i\}} \frac{|Du|^2}{2}\dd x &\leq  
   \sum_i \int_{\eps}^{R_i} \left (1+\left (1+\frac{L_i(\lambda)^2}{r^2}\right )\right ) \cdot \pi r \dd r.
  \end{align}
\end{proof}

In order to `glue' these symmetric independent cavitations, 
we build  an incompressible deformation 
far from the cavities using the flow of Dacorogna \& Moser \cite{DaMo90}
and the fine estimates of the previous section.

\begin{theorem}
 \label{th:scission-away}
 Let $n\in \N$ and ${\mathcal{B}}=B(0,R_0)\subset \R^2$. Suppose that the configuration 
  $\Big ( (a_i)_{i=1}^n, (v_i)_{i=1}^n \Big )$ is attainable. There exists $u_{ext} \in H^1( \mathcal{B}\setminus \bigcup_1^n B(a_i, R_i), \R^2)$, where the $R_i$ are as in 
 \eqref{eq:defRiri}, satisfying $u_{ext}(x)=\lambda x$ on $\partial \mathcal{B}$; $\det Du_{ext}\equiv 1$ in $ 
 \mathcal{B}\setminus \bigcup_1^n B(a_i, R_i)$; condition (INV); and 
 $$u_{ext}(a_i+R_ie^{i\theta})= z_i(\lambda) + \sqrt{L_i(\lambda)^2+R_i^2} e^{i\theta},
 \quad \forall\,i\in\{i,\ldots, n\}\ \forall\,\theta\in [0,2\pi].$$
\end{theorem}

Theorem \ref{th:main}  follows by combining the above results.

\begin{proof}[Proof of Theorem \ref{th:scission-away}]

We proceed as follows:
\begin{itemize}
 \item We fix the notation to describe the growth of the  (boundaries of the) circular holes
  (corresponding to the disks $B(a_i, R_i)$ of Proposition \ref{pr:near} which are not analyzed in 
  Theorem \ref{th:scission-away} and are, thus, removed from $\mathcal B$).
 \item At each instant we build a velocity field for the material points 
 by superposing two auxiliary fields, one that increases the radii $r_i(t)$ of the excised holes 
  and another that deals with the evolution of their centers $z_i(t)$.
 \item The trajectory of each material point is obtained as the solution of the ODE
 that establishes its relation to the previously constructed instantaneous velocity fields.
 \item We explain why the resulting deformation is injective and incompressible.
\end{itemize}

\subsubsection*{Evolution of the domains}
  For every $t\in [1,\lambda]$ set $$E(t):= B(0,tR_0) \setminus \bigcup_{i=1}^n B(z_i(t), r_i(t))$$
  where $r_i(t)$ is defined in \eqref{eq:defRiri}.
  By continuity, there exists $d>0$ (independent of $t$) such that 
\eqref{property} is satisfied, for every $t\in [1,\lambda]$,
with $z_i$ replaced with $z_i(t)$ and  $r_i$ replaced with $r_i(t)$.
Regarding $r_0(t)=tR_0$, note that $r_0(t)\leq \lambda R_0$ for all $t\in [1,\lambda]$.
Hence, setting $\delta:=\frac{d}{2\lambda R_0}$ (which depends on $n$, $R_0$,
$(a_i)_{i=1}^n$ and $(v_i)_{i=1}^n$ but not on $t$) we have that
$$E(t)\in \mathcal F_{\delta}\quad \forall\,t\in[1,\lambda].$$
In particular, by Theorem \ref{thm 2} there exists $C$ such that
$C_P\Big (E(t)\Big ) \leq C\cdot r_0(t)$ for all $t$. This implies that
$B\Big (E(t)\Big )\leq C$ for some $C$ independent of $t$, where $B\Big (E(t)\Big )$
is that of Proposition \ref{prop12}.

  \subsubsection*{A velocity field that accounts for the increase in the radii $r_i(t)$}
  
  Consider a fixed $t\in [1,\lambda]$. Define $g:\partial E(t)\to \R$ by
  $$g(y)=\frac{\dd r_i(t)}{\dd t}\quad \forall\,y\in \partial B(z_i(t), r_i(t)),\quad i\in \{0,1,\ldots, n\}.$$
  Clearly \eqref{eq:incLi} and \eqref{eq:defRiri} imply \eqref{nec}.
  We have thus all the hypotheses of Theorem \ref{thm3},
  which yields the existence of $v_t\in C^{2,\alpha}(\overline{E(t)}, \R^2)$
  such that 
  \begin{align}
   \label{eq:vel1}
   & \div v_t \equiv 0 \text{ in } E(t)\\
   & v_t\Big (z_i(t) + r_i(t)e^{i\theta} \Big ) = \frac{\dd r_i(t)}{\dd t} e^{i\theta} \quad \forall\,i,\theta \\
   & \Vert Dv_t \Vert_\infty \leq C\|g\|_\infty,
  \end{align}
  where $C=C\big (n, R_0, (a_i)_{i=1}^n, (v_i)_{i=1}^n\big )$.
  Recall that $L_i^2\in C^1([1,\lambda],[0,\infty))$ (by Definition \ref{de:attainable}), so
  $$\|g\|_\infty =\max_i \left |\frac{\frac{\dd}{\dd t} (L_i^2(t))}{r_i(t)}\right |
  \leq \frac{C}{\min_i R_i}$$
  is bounded above indepedently of $t$.

  \subsubsection*{A velocity field for the translation of the excised holes}
  
  Let $\eta \in C_c^\infty([0,1))$ be such that $\eta(0)=1$ and $\eta'(0)=0$.
  Define 
  $$ w(y):=\begin{cases}
            \eta \left ( \frac{r-r_i(t)}{d}\right ) \frac{\dd z_i(t)}{\dd t} \cdot (rie^{i\theta}),
            & \text{if } y=z_i(t) + re^{i\theta},\ r_i(t)\leq r < r_i(t) + d;
            \\
            0 & \text{in other case}
           \end{cases}
  $$
  and
  $$ {\tilde v}_{t}(y):=D^\perp w(y),\qquad y\in \overline{E(t)}.$$
  Then 
  \begin{align}
    & \div {\tilde v}_{t}\equiv 0\text{ in } E(t)\\
    & {\tilde v}_{t}(y) = \frac{\dd z_i(t)}{\dd t}\text{ on } \partial B(z_i(t), r_i(t))
  \end{align}
  and
  \begin{align*}
    \|D{\tilde v}_t\|_\infty &=\max_i \Bigg \|
    \left ( d^{-2} \eta'' e^{i\theta}\otimes e^{i\theta}     
    + (dr)^{-1}\eta' ie^{i\theta}\otimes ie^{i\theta}\right ) \frac{\dd z_i(t)}{\dd t} \cdot (rie^{i\theta})
    \\ &
    \hspace{18em}
      + d^{-1}\eta' \left (  \frac{\dd z_i(t)}{\dd t}\right )^\perp \otimes e^{i\theta}
    \Bigg \|_\infty
    \\
    &\leq C(d^{-2}\cdot (\lambda R_0)+d^{-1}) \left |\frac{\dd z_i(t)}{\dd t}\right |,
  \end{align*}
  which again is bounded uniformly in $t$ since $z_i\in C^1([1,\lambda],\R^2)$.
  
  \subsubsection*{Definition of $u_{ext}$ and energy bounds}
  
  For every $x\in \mathcal{B}\setminus \bigcup_1^n B(a_i, R_i)$ and every $t\in [1,\lambda]$ 
  let
  $f(x,t)$ be the solution of the Cauchy problem
  \begin{align}
   \begin{aligned}
     & \frac{\partial f}{\partial t} (x, t) = v_t(f(x,t)) + {\tilde v}_t(f(x,t))\\
     & f(x, 1) = x.
   \end{aligned}
  \end{align}
  It can be seen (as in Dacorogna \& Moser \cite{DaMo90})
  that the above autonomous ODE indeed
  has a well defined solution with enough regularity in time and space
  (in spite of the fact that the velocity fields are defined in changing domains).
  Moreover,
  $$ f(a_i + R_i e^{i\theta}, t) = z_i(t) + r_i(t) e^{i\theta}\quad \forall\, i, \theta$$
  and $$f(R_0 e^{i\theta}, t) = tR_0 e^{i\theta}$$
  thanks to the boundary conditions for $v_t$ and ${\tilde v}_t$.
  Define $u_{ext}$ by $$u_{ext}(x):= f(x, \lambda), \quad x\in   \mathcal{B}\setminus \bigcup_1^n B(a_i, R_i).$$
  For every $i\in\{i,\ldots, n\}$ and $\theta\in [0,2\pi]$
  $$ 
  u_{ext}(a_i+R_ie^{i\theta})= z_i(\lambda) + \sqrt{L_i(\lambda)^2+R_i^2} e^{i\theta}$$
  since $r_i(\lambda)= \sqrt{L_i(\lambda)^2+R_i^2}$. Also $u_{ext}(x)=\lambda x$ on $\partial \mathcal{B}$.
  
  The resulting deformation $u_{ext}$ is incompressible because
  \begin{align*}
    \frac{\partial}{\partial t} \det D_x f(x,t) &= \cof D_x f(x,t) \cdot D_x \frac{\partial f}{\partial t} (x,t )
    \\ &= \cof D_x f(x,t) \cdot D_x ((v_t+{\tilde v}_t)\circ f)(x,t)
    \\ &= \cof D_x f(x,t) \cdot (D_y(v_t+{\tilde v}_t)(f(x,t))D_x f(x,t))
    \\ &= (\cof D_x f(x,t)(D_x f(x,t))^T)\cdot D_y(v_t+{\tilde v}_t)(f(x,t))
    \\ &= (\det D_x f(x,t)) I\cdot D_y(v_t+{\tilde v}_t)(f(x,t))
  \end{align*}
  and the right-hand side is zero since $\div (v_t+{\tilde v}_t)\equiv 0$. 
  
  To see that $u_{ext}\in H^1$ it is enough to observe that
  \begin{align*}
   \frac{\dd}{\dd t} \int |D_xf(x,t)|^2\dd x
   &= \int D_xf(x,t)\cdot D_x \frac{\partial f}{\partial t} (x,t ) \dd x
   \\ &= \int D_xf(x,t)\cdot ((D_y(v_t+{\tilde v}_t))(f(x,t))D_xf(x,t)),
  \end{align*}
  whence 
  \begin{align*}
    \frac{\dd}{\dd t} \int |D_xf(x,t)|^2\dd x \leq 
    \underbrace{(\sup_t \|Dv_t+D{\tilde v}_t\|_{L^\infty(E(t))}  ) }_{:=C}
    \int |D_xf(x,t)|^2\dd x.
  \end{align*}
  This implies that $e^{-Ct} \int |D_xf(x,t)|^2$ decreases with $t$. Consequently,
  $$\int |Du_{ext}|^2 \leq e^{C(\lambda-1)} \int |I|^2\dd x < \infty.$$

   Finally, Ball's global invertibility theorem \cite{Ball81} shows that $u_{ext}$ is one-to-one a.e.\ 
   which combined with the previous energy estimate and \cite[Lemma 5.1]{BHM17}
   yields that $u_{ext}$ satisfies condition INV.

\end{proof}

\bibliography{tesis} \bibliographystyle{alpha}

\end{document}